\documentclass[11pt]{article}%
\usepackage{amssymb}
\usepackage{amsmath}
\usepackage{amsfonts}
\usepackage{graphicx}%
\newtheorem{theorem}{Theorem}

\newtheorem{corollary}[theorem]{Corollary}

\newtheorem{definition}[theorem]{Definition}
\newtheorem{example}[theorem]{Example}

\newtheorem{proposition}[theorem]{Proposition}
\newtheorem{remark}[theorem]{Remark}

\newenvironment{proof}[1][Proof]{\noindent\textbf{#1.} }{\ \rule{0.5em}{0.5em}}
\begin{document}

\title{Similar and Self-similar Curves in Minkowski n-space }
\author{Hakan Sim\c{s}ek
\and Mustafa \"{O}zdemir}
\maketitle

\begin{abstract}
In this paper, we investigate the similarity transformations in the
Minkowski-n space. We study the geometric invariants of non-null curves under
the similarity transformations. Besides, we extend the fundamental theorem for
a non-null curve according to a similarity motion of $\mathbb{E}_{1}^{n}$. We
determine all non-null self-similar curves in $\mathbb{E}_{1}^{n}$.

\qquad\ 



\end{abstract}

\section{\textbf{Introduction}}

\qquad A similarity transformation (or similitude) of Euclidean space, which
consists of a rotation, a translation and an isotropic scaling, is an
automorphism preserving the angles and ratios between lengths. The geometric
properties unchanged by similarity transformations is called the
\textit{similarity geometry}. The whole Euclidean geometry can be considered
as a glass of similarity geometry. The similarity transformations are studying
in most area of the pure and applied mathematics.

\qquad Curve matching is an important research area in the computer vision and
pattern recognition, which can help us determine what category the given test
curve belongs to. Also, the recognition and pose determination of 3D objects
can be represented by space curves is important for industry automation,
robotics, navigation and medical applications. S. Li \cite{vision2} showed an
invariant representation based on so-called similarity-invariant coordinate
system (SICS) for matching 3D space curves under the group of similarity
transformations. He also \cite{vision3} presented a system for matching and
pose estimation of 3D space curves under the similarity transformation. Brook
et al. \cite{vision00} discussed various problems of image processing and
analysis by using the similarity transformation. Sahbi \cite{vision0}
investigated a method for shape description based on kernel principal
component analysis (KPCA) in the similarity invariance of KPCA. There are many
applications of the similarity transformation in the computer vision and
pattern recognition (see also \cite{vision01, vision1}).

\qquad The idea of self-similarity is one of the most basic and fruitful ideas
in mathematics. A self-similar object is exactly similar to a part of itself,
which in turn remains similar to a smaller part of itself, and so on. In the
last few decades it established itself as the central notion in areas such as
fractal geometry, dynamical systems, computer networks and statistical
physics. Recently, the self-similarity started playing a role in algebra as
well, first of all in group theory (\cite{fractals, fractal4}). Mandelbrot
presented the first description of self-similar sets, namely sets that may be
expressed as unions of rescaled copies of themselves. He called these sets
fractals, which are systems that present such self-similar behavior and the
examples in nature are many. The Cantor set, the von Koch snowflake curve and
the Sierpinski gasket are some of the most famous examples of such sets.
Hutchinson and, shortly thereafter, Barnsley and Demko showed how systems of
contractive maps with associated probabilities, referred to as Iterated
Function Systems (IFS), can be used to construct fractal, self-similar sets
and measures supported on such sets (see \cite{fractal, fractal1, fractal2,
fractal3, fractal30}).

\qquad When the n-dimensional Euclidean space $\mathbb{E}^{n}$ is endowed with
the Lorentzian inner product, we obtain the \textit{Lorentzian similarity
geometry. }The Lorentzian flat geometry\textit{ }is inside the\textit{
}Lorentzian similarity geometr\textit{y. }Aristide \cite{simlorentz0}
investigated the closed\textit{ }Lorentzian similarity manifolds.\textit{\ }%
Kamishima \cite{simlorentz}\textit{ }studied the properties of compact
Lorentzian similarity manifolds using developing maps and holonomy
representations. The geometric invariants of curves in the Lorentzian
similarity geometry have not been considered so far. The theme of similarity
and self-similarity will be interesting in the Lorentzian-Minkowski space.

\qquad Many integrable equations, like Korteweg-de Vries (mKdV), sine-Gordon
and nonlinear Schr\"{o}dinger (NLS) equations, in soliton theory have been
shown to be related to motions of inextensible curves in the Euclidean space.
By using the similarity invariants of curves under the similarity motion, KS.
Chou and C. Qu \cite{chaos} showed that the motions of curves in two-, three-
and n-dimensional $(n>3)$ similarity geometries correspond to the Burgers
hierarchy, Burgers-mKdV hierarchy and a multi-component generalization of
these hierarchies in $\mathbb{E}^{n}$. Moreover, to study the motion of curves
in the Minkowski space also attracted researchers' interest. G\"{u}rses
\cite{gurses} studied the motion of curves on two-dimensional surface in
Minkowski 3-space. Q. Ding and J. Inoguchi \cite{inoguchi} showed that
binormal motions of curves in Minkowski 3-space are equivalent to some
integrable equations. In the 4-dimensional Minkowski space Nakayama
\cite{nakayama} presented a formulation on the motion of curves in
hyperboloids, which includes many equations integrable by means of the 1+1-
dimensional AKNS inverse scattering scheme. Therefore, the current paper can
contribute to study the motion of curves with similarity invariants in
$\mathbb{E}_{1}^{n}$.

\qquad Berger \cite{berger} represented the broad content of similarity
transformations in the arbitrary dimensional Euclidean spaces. Encheva
and\textsc{ }Georgiev \cite{shape, similar} studied the differential geometric
invariants of curves according to a similarity in the finite dimensional
Euclidean spaces. In the current paper, Lorentzian version of similarity
transformations will be entitled by pseudo-similarity transformation defined
by $\left(  \ref{05}\right)  $ in the section 2. The main idea of this paper
is to extend the fundamental theorem for a non-null curve with respect to
p-similarity motion and determine non-null self-similar curves in the
Minkowski n-space $\mathbb{E}_{1}^{n}$.

\qquad The scope of paper is as follows. First, we prove that the p-similarity
transformations preserve the causal characters and angles. We introduce
differential geometric invariants of a non-lightlike Frenet curve which are
called p-shape curvatures according to the group of p-similarity
transformations in $\mathbb{E}_{1}^{n}$. We give the uniqueness theorem which
states that two non-null curves having same the p-shape curvatures are
equivalent according to a p-similarity. Furthermore, we obtain the existence
theorem that is a process for constructing a non-null curve by its p-shape
curvatures under some initial conditions. Lastly, we study a exact description
of all non-null self-similar curves in $\mathbb{E}_{1}^{n}$. Especially, we
examine the low-dimensional cases $n=2,3,4$ in more details.

\section{Fundamental Group of the Lorentzian Similarity Geometry}

\qquad Firstly let us give some basic notions of the Lorentzian geometry. Let
$\mathbf{x}=\left(  x_{1},\cdots,x_{n}\right)  ^{T},$ $\mathbf{y}=\left(
y_{1},\cdots,y_{n}\right)  ^{T}$ and $\mathbf{z}=\left(  z_{1},\cdots
,z_{n}\right)  ^{T}$ be three arbitrary vectors in the Minkowski space
$\mathbb{E}_{1}^{n}.$ The Lorentzian inner product of $\mathbf{x}$ and
$\mathbf{y}$ can be stated as $\mathbf{x}\cdot\mathbf{y}=\mathbf{x}^{T}%
I^{\ast}\mathbf{y}$ where $I^{\ast}=diag(-1,1,1,\cdots,1).$ Then, the norm of
the vector $\mathbf{x}$ is represented by $\left\Vert \mathbf{x}\right\Vert
=\sqrt{\left\vert \mathbf{x}\cdot\mathbf{x}\right\vert },$ \cite{semi riemann,
grub}$.$

\begin{theorem}
\label{min}Let $\mathbf{x}$ and $\mathbf{y}$ be vectors in the Minkowski
n-space $\mathbb{E}_{1}^{n}$.

$\left(  i\right)  $ If $\mathbf{x}$ and $\mathbf{y}$ are timelike vectors
which are in the same timecone of $\mathbb{E}_{1}^{n}$ , then there is a
unique number $\theta\geq0,$ called the hyperbolic angle between $\mathbf{x}$
and $\mathbf{y}$ such that $\mathbf{x\cdot y}=-\left\Vert \mathbf{x}%
\right\Vert \left\Vert \mathbf{y}\right\Vert \cosh\theta$ $\mathbf{.}$

$\left(  ii\right)  $ If $\mathbf{x}$ and $\mathbf{y}$ are spacelike vectors
satisfying the inequality $\left\vert \mathbf{x\cdot y}\right\vert <\left\Vert
\mathbf{x}\right\Vert \left\Vert \mathbf{y}\right\Vert ,$ then $\mathbf{x\cdot
y}=\left\Vert \mathbf{x}\right\Vert \left\Vert \mathbf{y}\right\Vert
\cos\theta$ where $\theta$ is the angle between $\mathbf{x}$ and $\mathbf{y.}$

$\left(  iii\right)  $ If $\mathbf{x}$ and $\mathbf{y}$ are spacelike vectors
satisfying the inequality $\left\vert \mathbf{x\cdot y}\right\vert >\left\Vert
\mathbf{x}\right\Vert \left\Vert \mathbf{y}\right\Vert ,$ then $\mathbf{x\cdot
y}=\left\Vert \mathbf{x}\right\Vert \left\Vert \mathbf{y}\right\Vert
\cosh\theta$ where $\theta$ is the hyperbolic angle between $\mathbf{x}$ and
$\mathbf{y.}$
\end{theorem}

\qquad Now, we define similarity transformation in $\mathbb{E}_{1}^{n}$.
A\emph{ pseudo-similarity (p-similarity)} of Minkowski n-space $\mathbb{E}%
_{1}^{n}$ is a composition of a dilatation (homothety) and a Lorentzian
motion. Any p-similarity $f:\mathbb{E}_{1}^{n}\rightarrow\mathbb{E}_{1}^{n}$
is determined by%

\begin{equation}
f\left(  x\right)  =\mu\mathbf{A}x+\mathbf{b}, \label{05}%
\end{equation}
where $\mu$ is a real constant, $\mathbf{A}$ is a fixed pseudo-orthogonal
$n\times n$ matrix with $\det(\mathbf{A})=1$ and $\mathbf{b}=\left(
b_{1},...,b_{n}\right)  ^{T}\in%
\mathbb{R}
_{1}^{n}$ is a translation vector. When $n$ is odd and $\mu$ is a positive
real constant or $n$ is even and $\mu$ is a non-zero real constant, $f$ is an
orientation-preserving similarity transformation. When $n$ is odd and $\mu$ is
a negative real constant, $f$ is an orientation-reversing p-similarity
transformation. Since $f$ is a affine transformation, we get $\vec{f}\left(
\mathbf{u}\right)  =\mu\mathbf{Au}$ and $\left\Vert \vec{f}\left(
\mathbf{u}\right)  \right\Vert =\left\vert \mu\right\vert \left\Vert
\mathbf{u}\right\Vert $ for any $\mathbf{u}\in%
\mathbb{R}
_{1}^{n}$ where $\vec{f}\left(  \overrightarrow{xy}\right)  =\overrightarrow
{f(x)f(y)}$ (see \cite{berger}). The constant $\left\vert \mu\right\vert $ is
called p-similarity ratio of the transformation $f$. The p-similarity
transformations are a group under the composition of maps and we denote by
\textbf{Sim}$\left(  \mathbb{E}_{1}^{n}\right)  $. This group is a fundamental
group of the Lorentzian similarity geometry. Also, the group of
orientation-preserving (reversing) p-similarities are denoted by
\textbf{Sim}$^{+}\left(  \mathbb{E}_{1}^{n}\right)  $ (\textbf{Sim}%
$^{-}\left(  \mathbb{E}_{1}^{n}\right)  ,$ resp.).

\begin{theorem}
The p-similarity transformations preserve the causal characters and angles.
\end{theorem}

\begin{proof}
Let $f$ be a p-similarity. Then, since we can write the equation
\begin{equation}
\vec{f}(\mathbf{u})\cdot\vec{f}(\mathbf{u})=\mu^{2}\left(  \mathbf{Au\cdot
Au}\right)  =\mu^{2}\left(  \mathbf{u\cdot u}\right)  , \label{06}%
\end{equation}
$f$ preserves the causal character in $\mathbb{E}_{1}^{n}.$

\qquad Let $\mathbf{u}$ and $\mathbf{v}$ are timelike vectors which are in the
same timecone of $\mathbb{E}_{1}^{n}.$ We consider $\theta$ and $\gamma$ as
the angles between $\mathbf{u}$, $\mathbf{v}$ and $\vec{f}(\mathbf{u)}$,
$\vec{f}(\mathbf{v),}$ respectively. Since $\vec{f}(\mathbf{u)}$ and $\vec
{f}(\mathbf{v)}$ have same causal characters with $\mathbf{u}$ and
$\mathbf{v,}$ we can find the following equation from the Theorem $\ref{min};$%
\begin{align}
\vec{f}(\mathbf{u})\cdot\vec{f}(\mathbf{v})  &  =-\left\Vert \vec
{f}(\mathbf{u})\right\Vert \left\Vert \vec{f}(\mathbf{v})\right\Vert
\cosh\gamma\label{011}\\
\mu^{2}\left(  \mathbf{u\cdot v}\right)   &  =-\mu^{2}\left\Vert
\mathbf{u}\right\Vert \left\Vert \mathbf{v}\right\Vert \cosh\gamma\nonumber\\
-\left\Vert \mathbf{u}\right\Vert \left\Vert \mathbf{v}\right\Vert \cosh\theta
&  =-\left\Vert \mathbf{u}\right\Vert \left\Vert \mathbf{v}\right\Vert
\cosh\gamma\nonumber\\
\cosh\theta &  =\cosh\gamma.\nonumber
\end{align}
From here, we have $\theta=\gamma.$ If $\mathbf{u}$ and $\mathbf{v}$ are
spacelike vectors satisfying the inequality $\left\vert \mathbf{u\cdot
v}\right\vert <\left\Vert \mathbf{u}\right\Vert \left\Vert \mathbf{v}%
\right\Vert ,$ then
\[
\left\Vert \vec{f}(\mathbf{u})\right\Vert \left\Vert \vec{f}(\mathbf{v}%
)\right\Vert =\mu^{2}\left\Vert \mathbf{u}\right\Vert \left\Vert
\mathbf{v}\right\Vert >\mu^{2}\left\vert \mathbf{u\cdot v}\right\vert
=\left\vert \vec{f}(\mathbf{u})\cdot\vec{f}(\mathbf{v})\right\vert .
\]
Therefore, it can be said from the Theorem $\ref{min}$ that we have
$\theta=\gamma$ similar to $\left(  \ref{011}\right)  .$

\qquad It can also be found that $\theta$ is equal to $\gamma$ in case of
condition $\left(  iii\right)  $ in the Theorem $\ref{min}.$ As a consequence,
every p-similarity transformation preserves the angle between any two vectors.
\end{proof}

\section{Geometric Invariants of Non-null Curves in Lorentzian Similarity
Geometry}

\qquad Let $\alpha:t\in J\rightarrow\alpha\left(  t\right)  \in\mathbb{E}%
_{1}^{n}$ be a non-null curve of class $C^{n}$ and $\left\{  \mathbf{e}%
_{1},\cdots,\mathbf{e}_{n}\right\}  $ be a Frenet moving n-frame of $\alpha$
where $J\subset%
\mathbb{R}
$ is an open interval. We denote image of $\alpha$ under $f\in$ \textbf{Sim}%
$\left(  \mathbb{E}_{1}^{n}\right)  $ by $\alpha^{\ast}$ i.e. $\alpha^{\ast
}=f\circ\alpha.$ Then, $\alpha^{\ast}$ can be stated as
\begin{equation}
\alpha^{\ast}\left(  t\right)  =\mu\mathbf{A}\alpha\left(  t\right)  +b,\text{
\ \ \ \ \ \ }t\in I. \label{1}%
\end{equation}
The arc length functions of $\alpha$ and $\alpha^{\ast}$ starting at $t_{0}\in
J$ are%
\begin{equation}
s(t)=\int\limits_{t_{0}}^{t}\left\Vert \frac{d\alpha\left(  u\right)  }%
{du}\right\Vert du,\text{ \ \ \ \ \ \ }s^{\ast}\left(  t\right)
=\int\limits_{t_{0}}^{t}\left\Vert \frac{d\alpha^{\ast}\left(  u\right)  }%
{du}\right\Vert du=\left\vert \mu\right\vert s\left(  t\right)  . \label{1'}%
\end{equation}
In this section, we denote by a prime \textquotedblright$^{\prime}%
$\textquotedblright\ the differentiation with respect to $s$. The
i$^{\text{th}}$ curvature $\kappa_{i}$ of non-null curve $\alpha$ is given by
\begin{equation}
\kappa_{i}=\mathbf{e}_{i}^{\prime}\cdot\mathbf{e}_{i+1} \label{2}%
\end{equation}
for $1\leq i\leq n-1$ where $\kappa_{j}>0$ and $\kappa_{n-1}\neq0$,
$j=1,2,\cdots,n-2$. The Frenet-Serret equations of $\alpha$ in $\mathbb{E}%
_{1}^{n}$ is
\begin{align}
\mathbf{e}_{1}^{\prime}  &  =\varepsilon_{2}\kappa_{1}\mathbf{e}%
_{2}\nonumber\\
\mathbf{e}_{i}^{\prime}  &  =-\varepsilon_{i-1}\kappa_{i-1}\mathbf{e}%
_{i-1}+\varepsilon_{i+1}\kappa_{i}\mathbf{e}_{i+1},\text{ \ \ \ \ \ \ for
}2\leq i\leq n-1\label{f}\\
\mathbf{e}_{n}^{\prime}  &  =-\varepsilon_{n-1}\kappa_{n-1}\mathbf{e}%
_{n-1}\nonumber
\end{align}
where
\[
\varepsilon_{\ell}=\left\{
\begin{array}
[c]{lll}%
-1 & \text{if }\mathbf{e}_{\ell}\text{ is the timelike vector} & \\
\text{ \ }1 & \text{if }\mathbf{e}_{\ell}\text{ is the spacelike vector} &
\end{array}
\right.  \ \ \ \ \text{for }1\leq\ell\leq n.
\]

\qquad Let $\gamma\left(  s\right)  =\mathbf{e}_{1}\left(  s\right)  $ be the
spherical tangent indicatrix of $\alpha$ and $\sigma$ be an arc length
parameter of $\gamma$. it can be given a reparametrization of $\alpha$ by
$\sigma$
\begin{equation}
\alpha=\alpha\left(  \sigma\right)  :I\rightarrow\mathbb{E}_{1}^{n}, \label{3}%
\end{equation}
and the parameter $\sigma$ is called a spherical arc length parameter of
$\alpha.$ It is easily computed that
\begin{equation}
d\sigma=\kappa_{1}ds\text{ \ \ and \ \ }\frac{d\alpha}{d\sigma}=\frac
{1}{\kappa_{1}}\mathbf{e}_{1}. \label{4}%
\end{equation}
Thus, we can write
\begin{equation}
\frac{d}{d\sigma}%
\begin{bmatrix}
\mathbf{e}_{1}\\
\mathbf{e}_{2}\\
\mathbf{e}_{3}\\
\vdots\\
\mathbf{e}_{n-1}\\
\mathbf{e}_{n}%
\end{bmatrix}
=%
\begin{bmatrix}
0 & \varepsilon_{2} & 0 & \cdots & 0 & 0\\
-\varepsilon_{1} & 0 & \varepsilon_{3}\dfrac{\kappa_{2}}{\kappa_{1}} & \cdots
& 0 & 0\\
0 & -\varepsilon_{2}\dfrac{\kappa_{2}}{\kappa_{1}} & 0 & \cdots & 0 & 0\\
\vdots & \vdots & \ddots & \ddots & \vdots & \vdots\\
0 & 0 & 0 & \cdots & 0 & \varepsilon_{n}\dfrac{\kappa_{n-1}}{\kappa_{1}}\\
0 & 0 & 0 & \cdots & -\varepsilon_{n-1}\dfrac{\kappa_{n-1}}{\kappa_{1}} & 0
\end{bmatrix}%
\begin{bmatrix}
\mathbf{e}_{1}\\
\mathbf{e}_{2}\\
\mathbf{e}_{3}\\
\vdots\\
\mathbf{e}_{n-1}\\
\mathbf{e}_{n}%
\end{bmatrix}
. \label{5}%
\end{equation}

\qquad We consider the pseudo-orthogonal n-frame $\left\{  \frac{1}{\kappa
_{1}}\mathbf{e}_{1}\left(  \sigma\right)  ,\frac{1}{\kappa_{1}}\mathbf{e}%
_{2}\left(  \sigma\right)  ,\cdots,\frac{1}{\kappa_{1}}\mathbf{e}_{n}\left(
\sigma\right)  \right\}  ,$ $\sigma\in I,$ for the curve given by $\left(
\ref{3}\right)  .$ Let $\tilde{\kappa}_{1}$ denote the function $-\frac
{1}{\kappa_{1}}\frac{d\kappa_{1}}{d\sigma}$ and $\tilde{\kappa}_{i}$ denote
the function $\frac{\kappa_{i}}{\kappa_{1}}$ for $i=2,3,\cdots,n-1.$ Then,
using $\left(  \ref{f}\right)  ,$ $\left(  \ref{4}\right)  $ and $\left(
\ref{5}\right)  ,$ we get
\begin{equation}
\frac{d}{d\sigma}\left(  \frac{1}{\kappa_{1}}\mathbf{e}_{1}\left(
\sigma\right)  ,\frac{1}{\kappa_{1}}\mathbf{e}_{2}\left(  \sigma\right)
,\cdots,\frac{1}{\kappa_{1}}\mathbf{e}_{n}\left(  \sigma\right)  \right)
^{T}=\tilde{K}\left(  \frac{1}{\kappa_{1}}\mathbf{e}_{1}\left(  \sigma\right)
,\frac{1}{\kappa_{1}}\mathbf{e}_{2}\left(  \sigma\right)  ,\cdots,\frac
{1}{\kappa_{1}}\mathbf{e}_{n}\left(  \sigma\right)  \right)  ^{T} \label{f1}%
\end{equation}
where $\tilde{K}=%
\begin{bmatrix}
\tilde{\kappa}_{1} & \varepsilon_{2} & 0 & 0 & \cdots & 0 & 0 & 0\\
-\varepsilon_{1} & \tilde{\kappa}_{1} & \varepsilon_{3}\tilde{\kappa}_{2} &
0 & \cdots & 0 & 0 & 0\\
0 & -\varepsilon_{2}\tilde{\kappa}_{2} & \tilde{\kappa}_{1} & \varepsilon
_{4}\tilde{\kappa}_{3} & \cdots & 0 & 0 & 0\\
\vdots & \vdots & \ddots & \ddots & \ddots & \vdots & \vdots & \vdots\\
0 & 0 & 0 & 0 & \cdots & -\varepsilon_{n-2}\tilde{\kappa}_{n-2} &
\tilde{\kappa}_{1} & \varepsilon_{n}\tilde{\kappa}_{n-1}\\
0 & 0 & 0 & 0 & \cdots & 0 & -\varepsilon_{n-1}\tilde{\kappa}_{n-1} &
\tilde{\kappa}_{1}%
\end{bmatrix}
.$

\begin{proposition}
The pseudo-orthogonal frame
\begin{equation}
\frac{1}{\kappa_{1}}\mathbf{e}_{1}\left(  \sigma\right)  ,\frac{1}{\kappa_{1}%
}\mathbf{e}_{2}\left(  \sigma\right)  ,\cdots,\frac{1}{\kappa_{1}}%
\mathbf{e}_{n}\left(  \sigma\right)  ,\text{ \ \ \ \ \ \ \ \ }\sigma\in I
\label{8}%
\end{equation}
and the functions $\tilde{\kappa}_{1}=-\frac{1}{\kappa_{1}}\frac{d\kappa_{1}%
}{d\sigma},$ $\tilde{\kappa}_{i}=\frac{\kappa_{i}}{\kappa_{1}}$ $\left(
i=2,3,\cdots,n-1\right)  $ are invariant under the group of \textbf{Sim}%
$^{+}\left(  \mathbb{E}_{1}^{n}\right)  $ of the Minkowski n-space for the
non-null Frenet curve given by $\left(  \ref{3}\right)  $.
\end{proposition}

\begin{proof}
Let be $\mu>0$ when $n$ is odd and $\sigma^{\ast}$ be a spherical arc length
parameter of $\alpha^{\ast}.$ $\{\mathbf{e}_{1}^{\ast}=A\left(  \mathbf{e}%
_{1}\right)  ,\cdots,$ $\mathbf{e}_{n}^{\ast}=A\left(  \mathbf{e}_{n}\right)
\}$ is the Frenet-Serret frame of $\alpha^{\ast}.$ From $\left(
\ref{1}\right)  ,$ $\left(  \ref{1'}\right)  $ and $\left(  \ref{2}\right)  ,$
the i$^{\text{th}}$ curvature $\kappa_{i}^{\ast}$ of non-null curve
$\alpha^{\ast}$ can compute as follow%
\begin{equation}
\kappa_{i}^{\ast}\left(  s^{\ast}\right)  =\frac{d\mathbf{e}_{i}^{\ast}%
}{ds^{\ast}}\cdot\mathbf{e}_{i+1}^{\ast}=\frac{1}{\mu}\kappa_{i}\left(
s\right)  . \label{02}%
\end{equation}
We have
\begin{equation}
d\sigma^{\ast}=\kappa_{1}^{\ast}ds^{\ast}=\kappa_{1}ds=d\sigma\label{f3}%
\end{equation}
by using $\left(  \ref{02}\right)  $ and $\left(  \ref{4}\right)  .$ Then, we
get $\tilde{\kappa}_{1}^{\ast}=\tilde{\kappa}_{1}$ and $\tilde{\kappa}%
_{i}^{\ast}=\tilde{\kappa}_{i}$, i.e., the functions $\tilde{\kappa}_{i}$ are
invariant under the p-similarity transformation. Also, by definition of the
similarity transformation, we can write $\vec{f}\left(  \frac{1}{\kappa_{1}%
}\mathbf{e}_{1}\right)  =\frac{1}{\kappa_{1}^{\ast}}\mathbf{e}_{1}^{\ast}.$
Thus, the pseudo-orthogonal frame $\left(  \ref{8}\right)  $ is invariant
under the orientation-preserving p-similarity transformation. This proof also
is valid in the case of even $n$ and $\mu<0.$
\end{proof}

\begin{definition}
Let $\alpha:I\rightarrow\mathbb{E}_{1}^{n}$ be non-null curve of the class
$C^{n}$ parameterized by a spherical arc length parameter $\sigma.$ The
functions
\begin{equation}
\tilde{\kappa}_{1}=-\frac{1}{\kappa_{1}}\frac{d\kappa_{1}}{d\sigma}\text{
\ \ \ and \ \ \ \ }\tilde{\kappa}_{i}=\frac{\kappa_{i}}{\kappa_{1}}\text{
\ \ }i=2,3,\cdots,n-1 \label{000}%
\end{equation}
are called \emph{p-shape curvatures }of $\alpha$ in $\mathbb{E}_{1}^{n}.$
\end{definition}

\begin{remark}
The equation $\left(  \ref{f1}\right)  $ may be considered as the structure
equation in $\mathbb{E}_{1}^{n}$ with respect to the pseudo-orthogonal frame
$\left(  \ref{8}\right)  $ and the group \textbf{Sim}$^{+}\left(
\mathbb{E}_{1}^{n}\right)  .$
\end{remark}

\section{Fundamental Theorem of a Non-null Curve in Lorentzian Similarity
Geometry}

\qquad Two non-null space curves which have the same curvatures are always
equivalent according to Lorentzian motion. This notion can be extended with
respect to \textbf{Sim}$\left(  \mathbb{E}_{1}^{n}\right)  $ for the non-null
space curves which have the same p-shape curvatures in $\mathbb{E}_{1}^{n}.$

\begin{theorem}
\label{teklik}(Uniqueness Theorem) Let $\alpha,\alpha^{\ast}:I\rightarrow
\mathbb{E}_{1}^{n}$ be two non-null space curves of class $C^{n}$
parameterized by the same spherical arc length parameter $\sigma$ and have the
same causal character, where $I\subset%
\mathbb{R}
$ is an open interval. Suppose that $\alpha$ and $\alpha^{\ast}$ have the same
p-shape curvatures $\tilde{\kappa}_{i}=\tilde{\kappa}_{i}^{\ast}$ for any
$\sigma\in I,$ $i=1,2,\cdots,n-1.$

$\mathbf{i})$ If $n$ is odd and $\alpha,\alpha^{\ast}$ are the timelike
curves, there exists a $f\in$\textbf{Sim}$^{-}\left(  \mathbb{E}_{1}%
^{n}\right)  $ such that $\alpha^{\ast}=f\circ\alpha.$

$\mathbf{ii})$ If $n$ is odd and $\alpha,\alpha^{\ast}$ are the spacelike
curves, there exists a $f\in$\textbf{Sim}$^{+}\left(  \mathbb{E}_{1}%
^{n}\right)  $ such that $\alpha^{\ast}=f\circ\alpha.$

$\mathbf{iii})$ If $n$ is even, there exists a $f\in$\textbf{Sim}$^{+}\left(
\mathbb{E}_{1}^{n}\right)  $ such that $\alpha^{\ast}=f\circ\alpha.$
\end{theorem}

\begin{proof}
Let's the Lorentzian curvatures and an arc-length parameters of $\alpha$,
$\alpha^{\ast}$denoted by $\kappa_{i},$ $\kappa_{i}^{\ast}$ and $s,$ $s^{\ast
}$. Using the equality $\tilde{\kappa}_{1}=\tilde{\kappa}_{1}^{\ast}$ and
$\left(  \ref{000}\right)  ,$ we get $\kappa_{1}=\mu\kappa_{1}^{\ast}$ for
some real constant $\mu>0.$ Then, the equalities $\tilde{\kappa}_{i}%
=\tilde{\kappa}_{i}^{\ast}$ $\left(  i=2,3,\cdots,n-1\right)  $ imply
$\kappa_{i}=\mu\kappa_{i}^{\ast}.$ On the other hand, from $\left(
\ref{f3}\right)  $ we can write $ds=\frac{1}{\mu}ds^{\ast}.$

We can choose any point $\sigma_{0}\in I.$ There exists a Lorentzian motion
$\varphi$ of $\mathbb{E}_{1}^{n}$ such that
\[
\varphi\left(  \alpha\left(  \sigma_{0}\right)  \right)  =\alpha^{\ast}\left(
\sigma_{0}\right)  \text{ \ \ and \ \ }\varphi\left(  \mathbf{e}_{i}\left(
\sigma_{0}\right)  \right)  =-\varepsilon_{i}\mathbf{e}_{i}^{\ast}\left(
\sigma_{0}\right)  \text{ for }i=1,2,\cdots,n.
\]
Let's consider the function $\Psi:I\rightarrow%
\mathbb{R}
$ defined by
\[
\Psi\left(  \sigma\right)  =\left\Vert \varphi\left(  \mathbf{e}_{1}\left(
\sigma\right)  \right)  -\varepsilon_{1}\mathbf{e}_{1}^{\ast}\left(
\sigma\right)  \right\Vert ^{2}+\left\Vert \varphi\left(  \mathbf{e}%
_{2}\left(  \sigma\right)  \right)  -\varepsilon_{2}\mathbf{e}_{2}^{\ast
}\left(  \sigma\right)  \right\Vert ^{2}+\cdots+\left\Vert \varphi\left(
\mathbf{e}_{n}\left(  \sigma\right)  \right)  -\varepsilon_{n}\mathbf{e}%
_{n}^{\ast}\left(  \sigma\right)  \right\Vert ^{2}.
\]
Then
\begin{align*}
\frac{d\Psi}{d\sigma}  &  =2\left(  \frac{d}{d\sigma}\varphi\left(
\mathbf{e}_{1}\left(  \sigma\right)  \right)  -\varepsilon_{1}\frac{d}%
{d\sigma}\mathbf{e}_{1}^{\ast}\left(  \sigma\right)  \right)  \cdot\left(
\varphi\left(  \mathbf{e}_{1}\left(  \sigma\right)  \right)  -\varepsilon
_{1}\mathbf{e}_{1}^{\ast}\left(  \sigma\right)  \right) \\
&  +2\left(  \frac{d}{d\sigma}\varphi\left(  \mathbf{e}_{2}\left(
\sigma\right)  \right)  -\varepsilon_{2}\frac{d}{d\sigma}\mathbf{e}_{2}^{\ast
}\left(  \sigma\right)  \right)  \cdot\left(  \varphi\left(  \mathbf{e}%
_{2}\left(  \sigma\right)  \right)  -\varepsilon_{2}\mathbf{e}_{2}^{\ast
}\left(  \sigma\right)  \right) \\
&  +\cdots+2\left(  \frac{d}{d\sigma}\varphi\left(  \mathbf{e}_{n}\left(
\sigma\right)  \right)  -\varepsilon_{n}\frac{d}{d\sigma}\mathbf{e}_{3}^{\ast
}\left(  \sigma\right)  \right)  \cdot\left(  \varphi\left(  \mathbf{e}%
_{n}\left(  \sigma\right)  \right)  -\varepsilon_{n}\mathbf{e}_{n}^{\ast
}\left(  \sigma\right)  \right)  .
\end{align*}
Using $\left\Vert \varphi\left(  \mathbf{e}_{i}\right)  \right\Vert
^{2}=\left\Vert \mathbf{e}_{i}\right\Vert ^{2}=\left\Vert \mathbf{e}_{i}%
^{\ast}\right\Vert ^{2}=1$ we can write
\begin{align*}
\frac{d\Psi}{d\sigma}  &  =-2\varepsilon_{1}\left[  \left(  \varphi\left(
\frac{d}{d\sigma}\mathbf{e}_{1}\right)  \right)  \cdot\mathbf{e}_{1}^{\ast
}+\varphi\left(  \mathbf{e}_{1}\right)  \cdot\left(  \frac{d}{d\sigma
}\mathbf{e}_{1}^{\ast}\right)  \right] \\
&  -2\varepsilon_{2}\left[  \left(  \varphi\left(  \frac{d}{d\sigma}%
\mathbf{e}_{2}\right)  \right)  \cdot\mathbf{e}_{2}^{\ast}+\varphi\left(
\mathbf{e}_{2}\right)  \cdot\left(  \frac{d}{d\sigma}\mathbf{e}_{2}^{\ast
}\right)  \right] \\
&  -\cdots-2\varepsilon_{n}\left[  \left(  \varphi\left(  \frac{d}{d\sigma
}\mathbf{e}_{n}\right)  \right)  \cdot\mathbf{e}_{n}^{\ast}+\varphi\left(
\mathbf{e}_{n}\right)  \cdot\left(  \frac{d}{d\sigma}\mathbf{e}_{n}^{\ast
}\right)  \right]  .
\end{align*}
From $\left(  \ref{5}\right)  ,$ we get
\begin{align*}
\frac{d\Psi}{d\sigma}  &  =\left(  -2\varepsilon_{1}\varepsilon_{2}%
+2\varepsilon_{2}\varepsilon_{1}\right)  \left[  \varphi\left(  \mathbf{e}%
_{2}\right)  \cdot\mathbf{e}_{1}^{\ast}\right]  +\left(  -2\varepsilon
_{1}\varepsilon_{2}+2\varepsilon_{2}\varepsilon_{1}\right)  \left[
\varphi\left(  \mathbf{e}_{1}\right)  \cdot\mathbf{e}_{2}^{\ast}\right] \\
&  +\left(  -2\varepsilon_{2}\varepsilon_{3}\tilde{\kappa}_{2}+2\varepsilon
_{3}\varepsilon_{2}\tilde{\kappa}_{2}^{\ast}\right)  \left[  \varphi\left(
\mathbf{e}_{3}\right)  \cdot\mathbf{e}_{2}^{\ast}\right]  +\left(
-2\varepsilon_{2}\varepsilon_{3}\tilde{\kappa}_{2}^{\ast}+2\varepsilon
_{2}\varepsilon_{3}\tilde{\kappa}_{2}\right)  \left[  \varphi\left(
\mathbf{e}_{2}\right)  \cdot\mathbf{e}_{3}^{\ast}\right] \\
&  +\cdots+\left(  2\varepsilon_{n-1}\varepsilon_{n}\tilde{\kappa}%
_{n-1}-2\varepsilon_{n-1}\varepsilon_{n}\tilde{\kappa}_{n-1}^{\ast}\right)
\left[  \varphi\left(  \mathbf{e}_{n-1}\right)  \cdot\mathbf{e}_{n}^{\ast
}\right] \\
&  +\left(  2\varepsilon_{n-1}\varepsilon_{n}\tilde{\kappa}_{n-1}^{\ast
}-2\varepsilon_{n-1}\varepsilon_{n}\tilde{\kappa}_{n-1}\right)  \left[
\varphi\left(  \mathbf{e}_{n}\right)  \cdot\mathbf{e}_{n-1}^{\ast}\right]  .
\end{align*}
Since we have $\tilde{\kappa}_{i}=\tilde{\kappa}_{i}^{\ast},$ we find
$\dfrac{d\Psi}{d\sigma}=0$ for any $\sigma\in I.$ On the other hand, we know
$\Psi\left(  \sigma_{0}\right)  =0$ and thus we can write $\Psi\left(
\sigma\right)  =0$ for any $\sigma\in I.$ As a result, we can say that
\begin{equation}
\varphi\left(  \mathbf{e}_{i}\left(  \sigma\right)  \right)  =\varepsilon
_{i}\mathbf{e}_{i}^{\ast}\left(  \sigma\right)  ,\text{ \ \ \ \ \ }%
\forall\sigma\in I,\text{ \ \ }i=1,2,...,n. \label{10}%
\end{equation}

\qquad The map $g=\mu\varphi:\mathbb{E}_{1}^{n}\rightarrow\mathbb{E}_{1}^{n}$
is a p-similarity of $\mathbb{E}_{1}^{n}$. We examine an other function
$\Phi:I\rightarrow%
\mathbb{R}
$ such that
\[
\Phi\left(  \sigma\right)  =\left\Vert \frac{d}{d\sigma}g\left(  \alpha\left(
\sigma\right)  \right)  -\varepsilon_{1}\frac{d}{d\sigma}\alpha^{\ast}\left(
\sigma\right)  \right\Vert ^{2}\text{ \ \ \ for }\forall\sigma\in I.
\]
Taking derivative of this function with respect to $\sigma$ we get
\begin{align*}
\frac{d\Phi}{d\sigma}  &  =2g\left(  \frac{d^{2}\alpha}{d\sigma^{2}}\right)
\cdot g\left(  \frac{d\alpha}{d\sigma}\right)  -2\varepsilon_{1}\left[
g\left(  \frac{d^{2}\alpha}{d\sigma^{2}}\right)  \cdot\frac{d\alpha^{\ast}%
}{d\sigma}\right] \\
&  -\varepsilon_{1}2\frac{d^{2}\alpha^{\ast}}{d\sigma^{2}}\cdot g\left(
\frac{d\alpha}{d\sigma}\right)  +2\left[  \frac{d^{2}\alpha^{\ast}}%
{d\sigma^{2}}\cdot\frac{d\alpha^{\ast}}{d\sigma}\right]  .
\end{align*}
Since the function $\varphi$ is linear map, we can write by $\left(
\ref{4}\right)  $ and $\left(  \ref{10}\right)  $ the following equation
\[
\frac{d\Phi}{d\sigma}=2\varepsilon_{1}^{\ast}\mu^{2}\frac{\tilde{\kappa}_{1}%
}{\kappa_{1}^{2}}-2\varepsilon_{1}^{\ast}\mu\frac{\tilde{\kappa}_{1}}%
{\kappa_{1}\kappa_{1}^{\ast}}-2\varepsilon_{1}^{\ast}\mu\frac{\tilde{\kappa
}_{1}^{\ast}}{\kappa_{1}\kappa_{1}^{\ast}}+2\varepsilon_{1}^{\ast}\frac
{\tilde{\kappa}_{1}^{\ast}}{\left(  \kappa_{1}^{\ast}\right)  ^{2}}.
\]
Using $\mu=\dfrac{\kappa_{1}}{\kappa_{1}^{\ast}}$ it is obtained $\dfrac
{d\Phi}{d\sigma}=0.$ Also, it can be found
\[
\frac{d}{d\sigma}g\left(  \alpha\left(  \sigma_{0}\right)  \right)  =g\left(
\frac{1}{\kappa}\mathbf{e}_{1}\left(  \sigma_{0}\right)  \right)
=\varepsilon_{1}\frac{1}{\kappa^{\ast}}\mathbf{e}_{1}^{\ast}\left(  \sigma
_{0}\right)
\]
Then, we conclude that $\Phi\left(  \sigma_{0}\right)  =0$ from the equation
$\dfrac{d}{d\sigma}\alpha^{\ast}\left(  \sigma_{0}\right)  =\dfrac{1}%
{\kappa^{\ast}}\mathbf{e}_{1}^{\ast}\left(  \sigma_{0}\right)  .$ Hence, it
can be said $\Phi\left(  \sigma\right)  =0$ for $\forall\sigma\in I$. This
means that
\[
\frac{d}{d\sigma}g\left(  \alpha\left(  \sigma\right)  \right)  =\varepsilon
_{\mathbf{e}_{1}}\frac{d}{d\sigma}\alpha^{\ast}\left(  \sigma\right)
\]
or equivalently $\alpha^{\ast}\left(  \sigma\right)  =\varepsilon_{1}g\left(
\alpha\left(  \sigma\right)  \right)  +\mathbf{b}$ where $\mathbf{b}$ is a
constant vector. Then, the image of $\alpha$ under the p-similarity
$f=\xi\circ\left(  \varepsilon_{1}g\right)  $ is the non-null curve
$\alpha^{\ast}$, where $\xi:\mathbb{E}_{1}^{n}\rightarrow\mathbb{E}_{1}^{n}$
is a translation function determined by $\mathbf{b}.$ We consider that $n$ is
odd. If the curves $\alpha,\alpha^{\ast}$ are taken as the timelike curve,
then the p-similarity transformation $f$ is an orientation-reversing
transformation. Also, when the curves $\alpha,\alpha^{\ast}$ are the spacelike
curves, p-similarity transformation $f$ is an orientation-preserving transformation.
\end{proof}

\qquad The following theorem show that every $n-1$ functions of class
$C^{\infty}$ according to a p-similarity determine a non-null Frenet curve
under some initial conditions.

\begin{theorem}
\label{varl}(Existence Theorem) Let $z_{i}:I\rightarrow%
\mathbb{R}
,$ $i=1,2,\cdots,n-1$, be functions of class $C^{\infty}$ such that $z_{1},$
$z_{2},\cdots,$ $z_{n-1}$ have the same sign and $\mathbf{e}_{1}^{0},$
$\mathbf{e}_{2}^{0},\cdots,$ $\mathbf{e}_{n}^{0}$ be a pseodo-orthonormal
n-frame at a point $x_{0}$ in the Minkowski n-space. According to a
p-similarity with center $x_{0}$ there exists a unique non-null space curve
$\alpha:I\rightarrow\mathbb{E}_{1}^{n}$ parameterized by a spherical
arc-length parameter such that $\alpha$ satisfies the following conditions:

$\left(  i\right)  $ There exists $\sigma_{0}\in I$ such that $\alpha\left(
\sigma_{0}\right)  =x_{0}$ and the Frenet-Serret n-frame of $\alpha$ at
$x_{0}$ is $\mathbf{e}_{1}^{0},$ $\mathbf{e}_{2}^{0},\cdots,$ $\mathbf{e}%
_{n}^{0}.$

$\left(  ii\right)  $ $\tilde{\kappa}_{i}\left(  \sigma\right)  =z_{i}\left(
\sigma\right)  ,$ for any $\sigma\in I$ and $i=1,2,\cdots,n-1.$
\end{theorem}

\begin{proof}
Let us consider the following system of differential equations with respect to
a matrix-valued function $\mathbf{W}\left(  \sigma\right)  =\left(
\mathbf{e}_{1},\text{ }\mathbf{e}_{2},\cdots,\mathbf{e}_{n}\right)  ^{T}$%
\begin{equation}
\frac{d\mathbf{W}}{d\sigma}\left(  \sigma\right)  =\mathbf{M}\left(
\sigma\right)  \mathbf{W}\left(  \sigma\right)  \label{f4}%
\end{equation}
with a given matrix
\[
\mathbf{M}\left(  \sigma\right)  =%
\begin{bmatrix}
0 & \varepsilon_{2} & 0 & \cdots & 0 & 0 & 0\\
-\varepsilon_{1} & 0 & \varepsilon_{3}z_{2} & \cdots & 0 & 0 & 0\\
0 & -\varepsilon_{2}z_{2} & 0 & \ddots & 0 & 0 & 0\\
0 & 0 & -\varepsilon_{3}z_{3} & \ddots & 0 & 0 & 0\\
\vdots & \vdots & \vdots & \ddots & \ddots & \vdots & \vdots\\
0 & 0 & 0 & \cdots & -\varepsilon_{n-2}z_{n-2} & 0 & \varepsilon_{n}z_{n-1}\\
0 & 0 & 0 & \cdots & 0 & -\varepsilon_{n-1}z_{n-1} & 0
\end{bmatrix}
.
\]
The system $\left(  \ref{f4}\right)  $ has a unique solution $\mathbf{W}%
\left(  \sigma\right)  $ which satisfies the initial conditions $\mathbf{W}%
\left(  \sigma_{0}\right)  =\left(  \mathbf{e}_{1}^{0},\text{ }\mathbf{e}%
_{2}^{0},\cdots,\mathbf{e}_{n}^{0}\right)  ^{T}$ for $\sigma_{0}\in I.$ If
$\mathbf{W}^{t}\left(  \sigma\right)  $ is the transposed matrix of
$\mathbf{W}\left(  \sigma\right)  ,$ then
\begin{align*}
\frac{d}{d\sigma}\left(  \mathbf{I}^{\ast}\mathbf{W}^{t}\mathbf{I}^{\ast
}\mathbf{W}\right)   &  =\mathbf{I}^{\ast}\frac{d}{d\sigma}\mathbf{W}%
^{t}\mathbf{I}^{\ast}\mathbf{W}+\mathbf{I}^{\ast}\mathbf{W}^{t}\mathbf{I}%
^{\ast}\frac{d}{d\sigma}\mathbf{W}\\
&  =\mathbf{I}^{\ast}\mathbf{W}^{t}\mathbf{M}^{t}\mathbf{I}^{\ast}%
\mathbf{W}+\mathbf{I}^{\ast}\mathbf{W}^{t}\mathbf{I}^{\ast}\mathbf{MW}\\
&  =\mathbf{I}^{\ast}\mathbf{W}^{t}\left(  \mathbf{M}^{t}\mathbf{I}^{\ast
}+\mathbf{I}^{\ast}\mathbf{M}\right)  \mathbf{W}=0
\end{align*}
because of the equation $\mathbf{M}^{t}\mathbf{I}^{\ast}+\mathbf{I}^{\ast
}\mathbf{M}=%
\begin{bmatrix}
0
\end{bmatrix}
_{n\times n}$ where $\mathbf{I}^{\ast}=diag(\varepsilon_{1},\varepsilon
_{2},\cdots,\varepsilon_{n}).$ Also, we have $\mathbf{I}^{\ast}\mathbf{W}%
^{t}\left(  \sigma_{0}\right)  \mathbf{I}^{\ast}\mathbf{W}\left(  \sigma
_{0}\right)  =\mathbf{I}$ where $\mathbf{I}$ is the unit matrix since
$\left\{  \mathbf{e}_{1}^{0},\text{ }\mathbf{e}_{2}^{0},\cdots,\mathbf{e}%
_{n}^{0}\right\}  $ is the pseudo-orthonormal n-frame. As a result, we find
$\mathbf{I}^{\ast}\mathbf{X}^{t}\left(  \sigma\right)  \mathbf{I}^{\ast
}\mathbf{X}\left(  \sigma\right)  =\mathbf{I}$ for any $\sigma\in I.$ It means
that the vector fields $\left\{  \mathbf{e}_{1}^{0},\text{ }\mathbf{e}_{2}%
^{0},\cdots,\mathbf{e}_{n}^{0}\right\}  $ form pseudo-orthonormal frame field.

\qquad Let $\alpha:I\rightarrow\mathbb{E}_{1}^{n}$ be the regular non-null
curve given by
\begin{equation}
\alpha\left(  \sigma\right)  =x_{0}+\int_{\sigma_{0}}^{\sigma}e^{\int
z_{1}\left(  \sigma\right)  d\sigma}\mathbf{e}_{1}\left(  \sigma\right)
d\sigma,\text{ \ \ \ \ \ \ \ \ }\sigma\in I. \label{d}%
\end{equation}
By the equality $\left(  \ref{f4}\right)  $ and the linear independence of
$\left\{  \mathbf{e}_{1},\text{ }\mathbf{e}_{2},\cdots,\mathbf{e}_{n}\right\}
,$ we get that $\alpha\left(  \sigma\right)  $ is a non-null space curve in
$\mathbb{E}_{1}^{n}$ with p-shape curvatures $\tilde{\kappa}_{i}\left(
\sigma\right)  =z_{i}\left(  \sigma\right)  $ for $i=1,2,\cdots,n-1.$ Also,
the pseudo-orthonormal n-frame $\mathbf{e}_{1}\left(  \sigma\right)  ,$
$\mathbf{e}_{2}\left(  \sigma\right)  ,\cdots,\mathbf{e}_{n}\left(
\sigma\right)  $ is a Frenet-Serret n-frame of the non-null curve $\alpha.$
\end{proof}

\qquad By the Theorem $\ref{teklik}$ and $\ref{varl},$ we get the following theorem.

\begin{theorem}
\label{tek}Let $z_{i}:I\rightarrow%
\mathbb{R}
,$ $i=1,2,\cdots,n-1$, be the functions of class $C^{\infty}.$ According to
p-similarity there exists a unique non-null space curve with p-shape
curvatures $z_{i}.$
\end{theorem}

\begin{example}
Let p-shape curvatures $\left(  \tilde{\kappa}_{1},\tilde{\kappa}_{2}\right)
$ of the $\alpha:I\rightarrow\mathbb{E}_{1}^{3}$ be $\left(  0,a\right)  ,$
where $a\neq0$ is real constant, and the unit vector $\mathbf{e}_{2}\left(
\sigma\right)  $ be a timelike vector. Choose initial conditions
\begin{equation}
\mathbf{e}_{1}^{0}=\left(  0,-\frac{1}{\sqrt{1+a^{2}}},\frac{a}{\sqrt{1+a^{2}%
}}\right)  \mathbf{,}\text{ }\mathbf{\mathbf{e}}_{2}^{0}=\left(  1,0,0\right)
,\text{ }\mathbf{e}_{3}^{0}=\left(  0,\frac{a}{\sqrt{1+a^{2}}},\frac{1}%
{\sqrt{1+a^{2}}}\right)  . \label{i1}%
\end{equation}
Then, the system $\left(  \ref{f4}\right)  $ describes a spacelike vector
$\mathbf{e}_{1}$ defined by%
\begin{equation}
\mathbf{e}_{1}\left(  \sigma\right)  =\left(  \frac{1}{\sqrt{1+a^{2}}}%
\sinh\left(  \sqrt{1+a^{2}}\sigma\right)  ,-\frac{1}{\sqrt{1+a^{2}}}%
\cosh\left(  \sqrt{1+a^{2}}\sigma\right)  ,\frac{a}{\sqrt{1+a^{2}}}\right)
\label{i2}%
\end{equation}
with $\mathbf{e}_{1}\left(  0\right)  =\mathbf{e}_{1}^{0},$ in the Minkowski
3-space. Solving the equation $\left(  \ref{d}\right)  $ we obtain the
spacelike Frenet curve $\alpha$ (see figure 1) parameterized by
\begin{equation}
\alpha\left(  \sigma\right)  =\left(  \frac{1}{1+a^{2}}\cosh\left(
\sqrt{1+a^{2}}\sigma\right)  ,-\frac{1}{1+a^{2}}\sinh\left(  \sqrt{1+a^{2}%
}\sigma\right)  ,\frac{a}{\sqrt{1+a^{2}}}\sigma\right)  ,\ \ \sigma\in I.
\label{45}%
\end{equation}

\end{example}%

\begin{figure}
[ptb]
\begin{center}
\includegraphics[
height=3.1631in,
width=4.2in
]%
{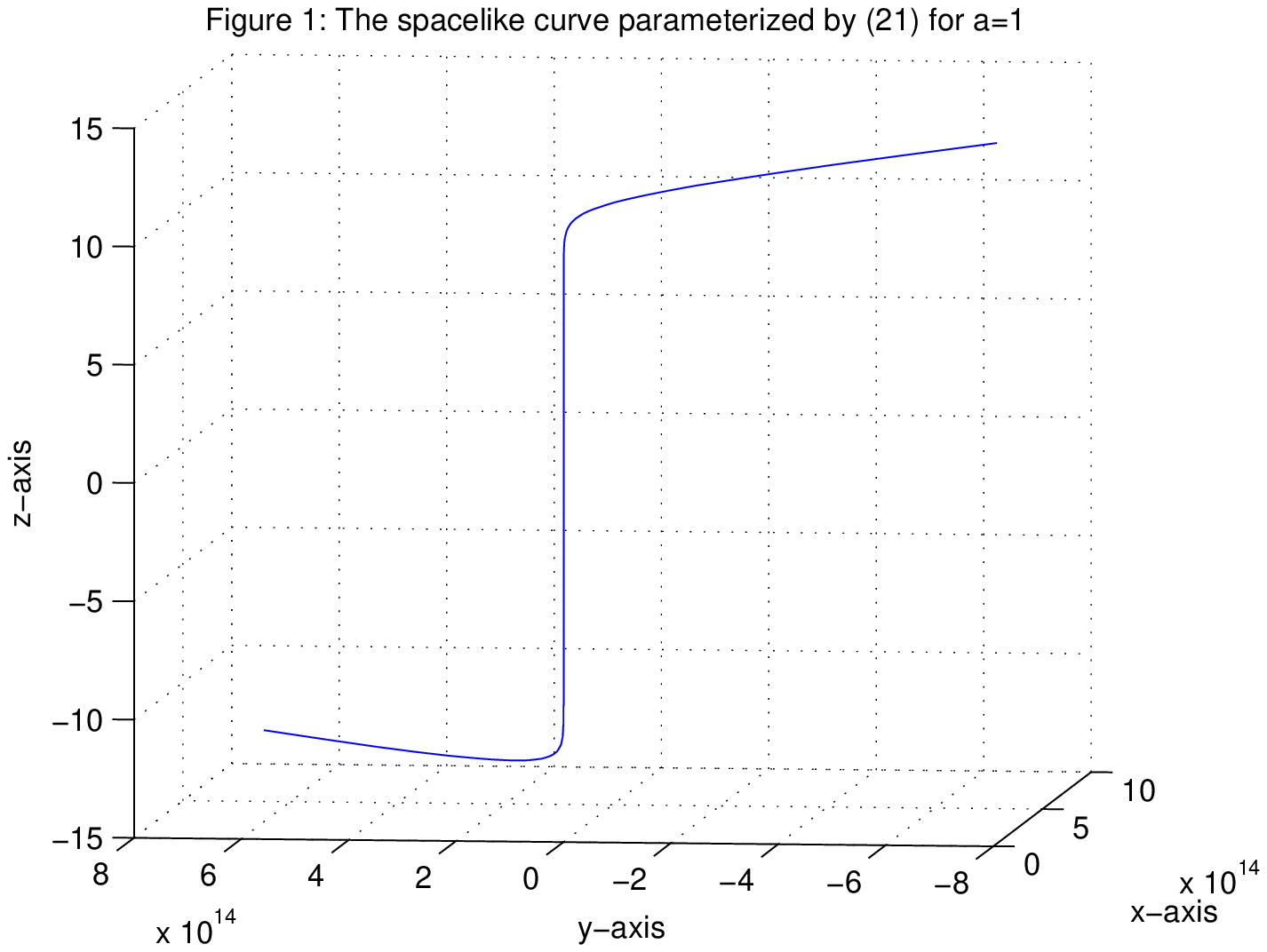}%
\end{center}
\end{figure}

\begin{example}
Let p-shape curvatures $\left(  \tilde{\kappa}_{1},\tilde{\kappa}_{2}%
,\tilde{\kappa}_{3}\right)  $ of the $\alpha:I\rightarrow\mathbb{E}_{1}^{4}$
be $\left(  \frac{1}{\sigma},0,0\right)  $ and the unit vector $\mathbf{e}%
_{1}\left(  \sigma\right)  $ be a timelike vector. Choose initial conditions
\[
\mathbf{e}_{1}^{0}=\left(  \sqrt{2},0,1,0\right)  \mathbf{,}\text{
}\mathbf{\mathbf{e}}_{2}^{0}=\left(  0,\frac{1}{\sqrt{2}},0,\frac{1}{\sqrt{2}%
}\right)  ,\text{ }\mathbf{e}_{3}^{0}=\left(  1,0,\sqrt{2},0\right)  ,\text{
}\mathbf{e}_{4}^{0}=\left(  0,-\frac{1}{\sqrt{2}},0,\frac{1}{\sqrt{2}}\right)
.
\]
Then, the system $\left(  \ref{f4}\right)  $ describes a timelike vector
$\mathbf{e}_{1}$ defined by%
\[
\mathbf{e}_{1}\left(  \sigma\right)  =\left(  \sqrt{2}\cosh\sigma,\frac
{\sinh\sigma}{\sqrt{2}},\cosh\sigma,\frac{\sinh\sigma}{\sqrt{2}}\right)
\]
with $\mathbf{e}_{1}\left(  0\right)  =\mathbf{e}_{1}^{0},$ in the Minkowski
space-time $\mathbb{E}_{1}^{4}$. Solving the equation $\left(  \ref{d}\right)
$ we obtain the timelike Frenet curve parameterized by
\begin{align*}
\alpha\left(  \sigma\right)   &  =(\sqrt{2}\left(  \sigma\sinh\sigma
-\cosh\sigma\right)  ,\frac{1}{\sqrt{2}}\left(  \sigma\cosh\sigma-\sinh
\sigma\right)  ,\\
&  \sigma\sinh\sigma-\cosh\sigma,\frac{1}{\sqrt{2}}\left(  \sigma\cosh
\sigma-\sinh\sigma\right)  )
\end{align*}
for any $\sigma\in I,$ where $\sigma$ is a spherical arc-length
parametrization of $\alpha$.
\end{example}

\section{The Relation between Focal Curvatures and p-shape Curvatures}

\qquad Let $\alpha:I\rightarrow\mathbb{E}_{1}^{n}$ be a unit speed non-null
space curve with the Frenet frame $\mathbf{e}_{1},$ $\mathbf{e}_{2}%
,\cdots,\mathbf{e}_{n}$ and let $s$ be an arc length parameter of $\alpha.$
The curve $\mathbf{\gamma:}I\rightarrow\mathbb{E}_{1}^{n}$ consisting of the
centers of the osculating hyperspheres of the curve $\alpha$ is called the
\emph{focal curve }of $\alpha.$ The focal curve can be represented by
\[
\mathbf{\gamma}\left(  s\right)  =\alpha\left(  s\right)  +m_{1}\left(
s\right)  \mathbf{e}_{2}+m_{2}\left(  s\right)  \mathbf{e}_{3}+\cdots
+m_{n-1}\left(  s\right)  \mathbf{e}_{n}%
\]
where $m_{1},\cdots,m_{n-1}$ are smooth functions called \emph{focal
curvatures }of $\alpha.$ Then, we have the following theorem from
\cite{minkowski}.

\begin{theorem}
\label{focal}The Euclidean curvatures of a non-null space curve $\alpha$ in
$\mathbb{E}_{1}^{n},$ parameterized by arc length, are given in terms of the
focal curvatures of $\alpha$ by the formula:%
\[
\kappa_{1}=\frac{\varepsilon_{1}}{m_{1}},\text{ \ \ \ \ \ }\kappa_{i}%
=\frac{\varepsilon_{2}m_{1}m_{1}^{^{\prime}}+\varepsilon_{3}m_{2}%
m_{2}^{^{\prime}}+\cdots+\varepsilon_{i}m_{i-1}m_{i-1}^{^{\prime}}}%
{m_{i-1}m_{i}},\text{ \ \ \ \ \ for }2\leq i\leq n.
\]

\end{theorem}

\qquad Now, we can restate all the p-shape curvatures $\tilde{\kappa}_{i}$ via
the focal curvatures and their derivatives.

\begin{proposition}
Let $\alpha:I\rightarrow\mathbb{E}_{1}^{n}$ be a unit speed non-null space
curve whose all Euclidean curvatures are non-zero. Then,
\begin{equation}
\tilde{\kappa}_{1}=\varepsilon_{1}m_{1}^{^{\prime}},\text{ \ \ \ }%
\tilde{\kappa}_{i}=\frac{\varepsilon_{1}m_{1}}{m_{i-1}m_{i}}\left(
\varepsilon_{2}m_{1}m_{1}^{^{\prime}}+\varepsilon_{3}m_{2}m_{2}^{^{\prime}%
}+\cdots+\varepsilon_{i}m_{i-1}m_{i-1}^{^{\prime}}\right)  ,\text{ \ \ for
}2\leq i\leq n. \label{foc}%
\end{equation}

\end{proposition}

\begin{proof}
By $\left(  \ref{000}\right)  $ we know
\begin{equation}
\tilde{\kappa}_{1}=\left(  \frac{1}{\kappa_{1}}\right)  ^{\prime
}\ \ \ \text{and}\ \ \ \tilde{\kappa}_{i}=\frac{\kappa_{i}}{\kappa_{1}%
}\ \ i=2,3,\cdots,n-1. \label{foc1}%
\end{equation}
So, It can be easily found the equations $\left(  \ref{foc}\right)  $ by using
the Theorem $\ref{focal}$ and $\left(  \ref{foc1}\right)  .$
\end{proof}

\section{The Non-null Self-similar Space Curves}

\qquad In this section, we study non-null self-similar space curves in the
Minkowski n-space $\mathbb{E}_{1}^{n}$. A non-null space curve $\alpha
:I\rightarrow\mathbb{E}_{1}^{n}$ is called \emph{self-similar }if any
p-similarity $f\in G$ conserve globally $\alpha$ and $G$ acts transitively on
$\alpha$ where $G$ is a one-parameter subgroup of \textbf{Sim}$\left(
\mathbb{E}_{1}^{n}\right)  .$ This means that all its the p-shape curvatures
$\tilde{\kappa}_{1},$ $\tilde{\kappa}_{2},\cdots,$ $\tilde{\kappa}_{n-1}$ are
constant. In fact, let $p=\alpha\left(  s_{1}\right)  $ and $q=\alpha\left(
s_{2}\right)  $ be two different points lying on $\alpha.$ Since $G$ acts
transitively on $\alpha,$ there is a similarity $f\in G$ such that $f\left(
p\right)  =q$ and then $\tilde{\kappa}_{i}\left(  s_{1}\right)  =\tilde
{\kappa}_{i}\left(  s_{2}\right)  $ because of the invariance of
$\tilde{\kappa}_{i},$ $i=1,2,\cdots,n-1.$ Every non-null self-similar curve
with the constant invariants $\tilde{\kappa}_{1}=0,$ $\tilde{\kappa}_{i}>0$
and $\tilde{\kappa}_{n-1}\neq0$ ,$i=2,3,\cdots,n-2,$ has the constant
Euclidean curvatures
\[
\kappa_{1}>0,\text{ }\kappa_{2}=\tilde{\kappa}_{2}\kappa_{1}>0,\cdots
,\kappa_{n-2}=\tilde{\kappa}_{n-2}\kappa_{1}>0,\text{ }\kappa_{n-1}%
=\tilde{\kappa}_{n-1}\kappa_{1}\neq0.
\]

\qquad We will investigate the non-null space curve with the constant p-shape
curvatures
\begin{equation}
\tilde{\kappa}_{1}\neq0,\tilde{\kappa}_{2}\neq0,\cdots,\tilde{\kappa}%
_{n-1}\neq0. \label{f6}%
\end{equation}
Consider the constant matrix
\[
M=%
\begin{bmatrix}
0 & \varepsilon_{2} & 0 & 0 & \cdots & 0 & 0 & 0\\
-\varepsilon_{1} & 0 & \varepsilon_{3}\tilde{\kappa}_{2} & 0 & \cdots & 0 &
0 & 0\\
0 & -\varepsilon_{2}\tilde{\kappa}_{2} & 0 & \varepsilon_{4}\tilde{\kappa}_{3}
& \cdots & 0 & 0 & 0\\
0 & 0 & -\varepsilon_{3}\tilde{\kappa}_{3} & 0 & \ddots & 0 & 0 & 0\\
\vdots & \vdots & \vdots & \ddots & \ddots & \vdots & \vdots & \vdots\\
0 & 0 & 0 & 0 & \cdots & -\varepsilon_{n-2}\tilde{\kappa}_{n-2} & 0 &
\varepsilon_{n}\tilde{\kappa}_{n-1}\\
0 & 0 & 0 & 0 & \cdots & 0 & -\varepsilon_{n-1}\tilde{\kappa}_{n-1} & 0
\end{bmatrix}
.
\]
The matrix $M^{2}$ is a semi skew-symmetrix matrix and has exactly
$k=\left\lfloor \frac{n}{2}\right\rfloor $ eigenvalues with multiplicity two:
$\lambda_{1}^{2},$ $\lambda_{2}^{2},\cdots,\lambda_{k}^{2}.$ According to
Corollary $3.3$ in \cite{normal}, the normal form of the matrix $M$ is either
\[%
\begin{bmatrix}
0 & \varepsilon_{2}\lambda_{1} & 0 & 0 & \cdots & 0 & 0\\
-\varepsilon_{1}\lambda_{1} & 0 & 0 & 0 & \cdots & 0 & 0\\
0 & 0 & 0 & \varepsilon_{4}\lambda_{2} & \cdots & 0 & 0\\
0 & 0 & -\varepsilon_{3}\lambda_{2} & 0 & \ddots & 0 & 0\\
\vdots & \vdots & \vdots & \ddots & \ddots & \vdots & \vdots\\
0 & 0 & 0 & 0 & \cdots & 0 & \varepsilon_{n}\lambda_{k}\\
0 & 0 & 0 & 0 & \cdots & -\varepsilon_{n-1}\lambda_{k} & 0
\end{bmatrix}
\]
in the case of even $n$, or the same matrix with an additional row (column) of
zeros in the case of odd $n$.

\subsection{Non-null Self-similar Curves in Even-dimensional Minkowski Space}

\qquad Any non-null self-similar curve in $\mathbb{E}_{1}^{2k}$ can be
described by its constant p-shape curvatures with respect to Theorem
$\ref{varl}.$ Let's see this via the following theorem.

\begin{theorem}
\label{cift self}Let $\alpha:I\rightarrow\mathbb{E}_{1}^{2k}$ be a non-null
self-similar curve with the constant p-shape curvatures $\tilde{\kappa}%
_{1}\neq0,$ $\tilde{\kappa}_{2}\neq0,\cdots,\tilde{\kappa}_{2k-1}\neq0.$
Suppose that $\lambda_{1}^{2},$ $\lambda_{2}^{2},\cdots,\lambda_{k}^{2}$ are
all different eigenvalues of the symmetric matrix $M^{2}.$ So,

$\mathbf{i})$ If the unit vector $\mathbf{e}_{1}$ is a timelike vector, then a
parametric statement of the timelike self-similar curve $\alpha$ according to
arc-length parameter $\sigma$ can be written in the form
\begin{align}
\alpha\left(  \sigma\right)   &  =(\frac{a_{1}}{b_{1}}e^{\tilde{\kappa}%
_{1}\sigma}\sinh\theta_{1},\frac{a_{1}}{b_{1}}e^{\tilde{\kappa}_{1}\sigma
}\cosh\theta_{1},\nonumber\\
&  \frac{a_{2}}{b_{2}}e^{\tilde{\kappa}_{1}\sigma}\sin\theta_{2},-\frac{a_{2}%
}{b_{2}}e^{\tilde{\kappa}_{1}\sigma}\cos\theta_{2},\cdots,\frac{a_{k}}{b_{k}%
}e^{\tilde{\kappa}_{1}\sigma}\sin\theta_{k},-\frac{a_{k}}{b_{k}}%
e^{\tilde{\kappa}_{1}\sigma}\cos\theta_{k}) \label{eq}%
\end{align}
where
\[
b_{1}=\sqrt{\lambda_{1}^{2}-\tilde{\kappa}_{1}^{2}},\text{ \ \ \ \ \ \ }%
\theta_{1}=\lambda_{1}\sigma-\cosh^{-1}\left(  \frac{\lambda_{1}}%
{\sqrt{\lambda_{1}^{2}-\tilde{\kappa}_{1}^{2}}}\right)  ,
\]
and for $i=2,\cdots,k$
\[
b_{i}=\sqrt{\lambda_{i}^{2}+\tilde{\kappa}_{1}^{2}},\text{ \ \ \ \ \ \ }%
\theta_{i}=\lambda_{i}\sigma+\cos^{-1}\left(  \frac{\lambda_{i}}{\sqrt
{\lambda_{i}^{2}+\tilde{\kappa}_{1}^{2}}}\right)  .
\]
The real different non-zero numbers $a_{1},\cdots,a_{k}$ are a solution of the
system of $k$ algebraic quadratic equations
\[
\mathbf{e}_{1}\cdot\mathbf{e}_{1}=-1\text{ \ \ and\ \ \ }\mathbf{e}_{i}%
\cdot\mathbf{e}_{i}=1\text{ \ \ \ }i=2,...,k,
\]
determined by the vectors
\begin{align}
\mathbf{e}_{1}\left(  \sigma\right)   &  =e^{-\tilde{\kappa}_{1}\sigma}%
\frac{d}{d\sigma}\alpha\left(  \sigma\right) \nonumber\\
\mathbf{e}_{2}\left(  \sigma\right)   &  =\frac{d}{d\sigma}\mathbf{e}%
_{1}\left(  \sigma\right) \nonumber\\
\mathbf{e}_{3}\left(  \sigma\right)   &  =\frac{1}{\tilde{\kappa}_{2}}\left(
-\mathbf{e}_{1}\left(  \sigma\right)  +\frac{d}{d\sigma}\mathbf{e}_{2}\left(
\sigma\right)  \right) \label{f11}\\
\mathbf{e}_{4}\left(  \sigma\right)   &  =\frac{1}{\tilde{\kappa}_{3}}\left(
\tilde{\kappa}_{2}\mathbf{e}_{2}\left(  \sigma\right)  +\frac{d}{d\sigma
}\mathbf{e}_{3}\left(  \sigma\right)  \right) \nonumber\\
&  \vdots\text{ \ \ \ \ \ \ \ \ \ \ \ \ \ \ \ }\vdots\nonumber\\
\mathbf{e}_{k}\left(  \sigma\right)   &  =\frac{1}{\tilde{\kappa}_{k-1}%
}\left(  \tilde{\kappa}_{k-2}\mathbf{e}_{k-2}\left(  \sigma\right)  +\frac
{d}{d\sigma}\mathbf{e}_{k-1}\left(  \sigma\right)  \right)  .\nonumber
\end{align}

$\mathbf{ii})$ If the unit vector $\mathbf{e}_{2}$ is a timelike vector, then
a parametric representation of the spacelike self-similar curve $\alpha$
according to arc-length parameter $\sigma$ can be written in the form
\begin{align}
\alpha\left(  \sigma\right)   &  =(-\frac{a_{1}}{b_{1}}e^{\tilde{\kappa}%
_{1}\sigma}\cosh\theta_{1},-\frac{a_{1}}{b_{1}}e^{\tilde{\kappa}_{1}\sigma
}\sinh\theta_{1},\nonumber\\
&  \frac{a_{2}}{b_{2}}e^{\tilde{\kappa}_{1}\sigma}\sin\theta_{2},-\frac{a_{2}%
}{b_{2}}e^{\tilde{\kappa}_{1}\sigma}\cos\theta_{2},\cdots,\frac{a_{k}}{b_{k}%
}e^{\tilde{\kappa}_{1}\sigma}\sin\theta_{k},-\frac{a_{k}}{b_{k}}%
e^{\tilde{\kappa}_{1}\sigma}\cos\theta_{k}) \label{spa}%
\end{align}
The real different non-zero numbers $a_{1},\cdots,a_{k}$ are a solution of the
system of $k$ algebraic quadratic equations
\[
\mathbf{e}_{2}\cdot\mathbf{e}_{2}=-1\text{ \ \ and\ \ \ }\mathbf{e}_{i}%
\cdot\mathbf{e}_{i}=1\text{ \ \ \ }i=1,3,4,\cdots,k,
\]
determined by the vectors
\begin{align}
\mathbf{e}_{1}\left(  \sigma\right)   &  =e^{-\tilde{\kappa}_{1}\sigma}%
\frac{d}{d\sigma}\alpha\left(  \sigma\right) \nonumber\\
\mathbf{e}_{2}\left(  \sigma\right)   &  =-\frac{d}{d\sigma}\mathbf{e}%
_{1}\left(  \sigma\right) \nonumber\\
\mathbf{e}_{3}\left(  \sigma\right)   &  =\frac{1}{\tilde{\kappa}_{2}}\left(
\mathbf{e}_{1}\left(  \sigma\right)  +\frac{d}{d\sigma}\mathbf{e}_{2}\left(
\sigma\right)  \right) \\
\mathbf{e}_{4}\left(  \sigma\right)   &  =\frac{1}{\tilde{\kappa}_{3}}\left(
-\tilde{\kappa}_{2}\mathbf{e}_{2}\left(  \sigma\right)  +\frac{d}{d\sigma
}\mathbf{e}_{3}\left(  \sigma\right)  \right) \nonumber\\
&  \vdots\text{ \ \ \ \ \ \ \ \ \ \ \ \ \ \ \ }\vdots\nonumber\\
\mathbf{e}_{k}\left(  \sigma\right)   &  =\frac{1}{\tilde{\kappa}_{k-1}%
}\left(  \tilde{\kappa}_{k-2}\mathbf{e}_{k-2}\left(  \sigma\right)  +\frac
{d}{d\sigma}\mathbf{e}_{k-1}\left(  \sigma\right)  \right)  .\nonumber
\end{align}

\end{theorem}

\begin{proof}
$i)$ The systems of unit vectors $\mathbf{e}_{1}\left(  \sigma\right)  ,$
$\mathbf{e}_{2}\left(  \sigma\right)  ,\cdots,$ $\mathbf{e}_{2k-1}\left(
\sigma\right)  ,$ $\mathbf{e}_{2k}\left(  \sigma\right)  $ can be expressed as
a solution of the ordinary differential equations
\[
\frac{d}{d\sigma}\omega=M\omega
\]
where $\omega\left(  \sigma\right)  =\left(  \mathbf{e}_{1}\left(
\sigma\right)  ,\mathbf{e}_{2}\left(  \sigma\right)  ,\cdots,\mathbf{e}%
_{2k}\left(  \sigma\right)  \right)  ^{T}$. Using the normal form of the
matrix $M,$ we have that the unit vector $\mathbf{e}_{1}$ is the equal to
\begin{align*}
\mathbf{e}_{1}\left(  \sigma\right)   &  =(a_{1}\cosh\left(  \lambda_{1}%
\sigma\right)  ,a_{1}\sinh\left(  \lambda_{1}\sigma\right)  ,\\
&  a_{2}\cos\left(  \lambda_{2}\sigma\right)  ,a_{2}\sin\left(  \lambda
_{2}\sigma\right)  ,\cdots,a_{k}\cos\left(  \lambda_{k}\sigma\right)
,a_{k}\sin\left(  \lambda_{k}\sigma\right)  )
\end{align*}
where $a_{i}$'s are a real constants satisfying $-a_{1}^{2}+\sum
\limits_{i=2}^{k}a_{i}^{2}=-1.$

\qquad If $X=\left(  \alpha_{1}\left(  \sigma\right)  ,\alpha_{2}\left(
\sigma\right)  ,\cdots,\alpha_{2k}\left(  \sigma\right)  \right)  $ is
considered as the parametric equation of the timelike curve $\alpha,$ we get
the following equation by $\left(  \ref{4}\right)  $
\begin{equation}
\frac{d}{d\sigma}X=\frac{1}{\kappa_{1}}\mathbf{e}_{1} \label{f10}%
\end{equation}
where we have $\kappa_{1}=e^{-\tilde{\kappa}_{1}\sigma}$ by $\tilde{\kappa
}_{1}=-\dfrac{1}{\kappa_{1}}\dfrac{d\kappa_{1}}{d\sigma}$. It can be obtained
the following equation by $\left(  \ref{f10}\right)  $
\begin{align*}
\alpha_{1}  &  =\frac{a_{1}}{\tilde{\kappa}_{1}}e^{\tilde{\kappa}_{1}\sigma
}\cosh\left(  \lambda_{1}\sigma\right)  -\frac{\lambda_{1}}{\tilde{\kappa}%
_{1}}\alpha_{2}\\
\alpha_{2}  &  =\frac{a_{1}}{\tilde{\kappa}_{1}}e^{\tilde{\kappa}_{1}\sigma
}\sinh\left(  \lambda_{1}\sigma\right)  -\frac{\lambda_{1}}{\tilde{\kappa}%
_{1}}\alpha_{1}%
\end{align*}
and
\begin{align*}
\alpha_{2i-1}  &  =\frac{a_{i}}{\tilde{\kappa}_{1}}e^{\tilde{\kappa}_{1}%
\sigma}\cosh\left(  \lambda_{i}\sigma\right)  +\frac{\lambda_{i}}%
{\tilde{\kappa}_{1}}\alpha_{2i}\\
\alpha_{2i}  &  =\frac{a_{i}}{\tilde{\kappa}_{1}}e^{\tilde{\kappa}_{1}\sigma
}\sinh\left(  \lambda_{i}\sigma\right)  -\frac{\lambda_{i}}{\tilde{\kappa}%
_{1}}\alpha_{2i-1}%
\end{align*}
for $i=2,\cdots,k.$ The solutions of above the linear systems are
\begin{align*}
\alpha_{1}  &  =\frac{a_{1}}{\tilde{\kappa}_{1}^{2}-\lambda_{1}^{2}}%
e^{\tilde{\kappa}_{1}\sigma}\left(  \tilde{\kappa}_{1}\cosh\left(  \lambda
_{1}\sigma\right)  -\lambda_{1}\sinh\left(  \lambda_{1}\sigma\right)  \right)
=\frac{a_{1}}{b_{1}}e^{\tilde{\kappa}_{1}\sigma}\sinh\theta_{1}\\
\alpha_{2}  &  =\frac{a_{1}}{\tilde{\kappa}_{1}^{2}-\lambda_{1}^{2}}%
e^{\tilde{\kappa}_{1}\sigma}\left(  \tilde{\kappa}_{1}\sinh\left(  \lambda
_{1}\sigma\right)  -\lambda_{1}\cosh\left(  \lambda_{1}\sigma\right)  \right)
=\frac{a_{1}}{b_{1}}e^{\tilde{\kappa}_{1}\sigma}\cosh\theta_{1}%
\end{align*}
and%
\begin{align*}
\alpha_{2i-1}  &  =\frac{a_{i}}{\tilde{\kappa}_{1}^{2}+\lambda_{i}^{2}%
}e^{\tilde{\kappa}_{1}\sigma}\left(  \tilde{\kappa}_{1}\cos\left(  \lambda
_{i}\sigma\right)  +\lambda_{i}\sin\left(  \lambda_{i}\sigma\right)  \right)
=\frac{a_{i}}{b_{i}}e^{\tilde{\kappa}_{1}\sigma}\sin\theta_{i}\\
\alpha_{2i}  &  =\frac{a_{1}}{\tilde{\kappa}_{1}^{2}-\lambda_{i}^{2}}%
e^{\tilde{\kappa}_{1}\sigma}\left(  \tilde{\kappa}_{1}\sin\left(  \lambda
_{i}\sigma\right)  -\lambda_{i}\cos\left(  \lambda_{i}\sigma\right)  \right)
=-\frac{a_{i}}{b_{i}}e^{\tilde{\kappa}_{1}\sigma}\cos\theta_{i}.
\end{align*}

\qquad We can get the unit vectors $\left(  \ref{f11}\right)  $ by using the
equations $\left(  \ref{f1}\right)  $ for the timelike self-similar curve
$\alpha.$ Thus, we can write the following equations by means of $\left(
\ref{f11}\right)  $%
\begin{align*}
\mathbf{e}_{1}\cdot\mathbf{e}_{1}  &  =-1\text{ \ }\Rightarrow\text{ \ }%
-a_{1}^{2}+\sum\limits_{i=2}^{k}a_{i}^{2}=-1\\
\mathbf{e}_{2}\cdot\mathbf{e}_{2}  &  =1\text{ \ \ }\Rightarrow\text{
\ \ }\sum\limits_{i=1}^{k}a_{i}^{2}\lambda_{i}^{2}=1\\
\mathbf{e}_{3}\cdot\mathbf{e}_{3}  &  =1\text{ \ \ }\Rightarrow\text{
\ \ }-a_{1}^{2}\left(  1-\lambda_{1}^{2}\right)  ^{2}+\sum\limits_{i=2}%
^{k}\left(  1+\lambda_{i}^{2}\right)  ^{2}a_{i}^{2}=\tilde{\kappa}_{2}^{2}\\
\mathbf{e}_{i}\cdot\mathbf{e}_{i}  &  =1\text{ \ \ }\Rightarrow\text{ \ so on,
}i=4,\cdots,k.
\end{align*}

$ii)$ The proof is similar to the proof of $i).$
\end{proof}

\begin{corollary}
$\mathbf{i})$ Let $\alpha:I\rightarrow\mathbb{E}_{1}^{2k}$ $(k>1)$ be a
timelike self-similar curve with a parametric representation $\left(
\ref{eq}\right)  .$ Then, this curve lies on the quadratic timelike
hypersurface with an equation%
\[
\frac{b_{1}^{2}}{a_{1}^{2}}\left(  -x_{1}^{2}+x_{2}^{2}\right)  +\frac
{b_{2}^{2}}{a_{2}^{2}}\left(  x_{1}^{2}+x_{2}^{2}\right)  +\cdots
+\frac{b_{k-1}^{2}}{a_{k-1}^{2}}\left(  x_{2k-3}^{2}+x_{2k-2}^{2}\right)
=(k-1)\frac{b_{k}^{2}}{a_{k}^{2}}\left(  x_{2k-1}^{2}+x_{2k}^{2}\right)  .
\]

$\mathbf{ii})$ Let $\alpha:I\rightarrow\mathbb{E}_{1}^{2k}$ $(k>1)$ be a
spacelike self-similar curve with a parametric representation $\left(
\ref{spa}\right)  .$ Then, this curve lies on the quadratic Lorentzian
hypersurface with an equation%
\[
\frac{b_{1}^{2}}{a_{1}^{2}}\left(  x_{1}^{2}-x_{2}^{2}\right)  +\frac
{b_{2}^{2}}{a_{2}^{2}}\left(  x_{1}^{2}+x_{2}^{2}\right)  +\cdots
+\frac{b_{k-1}^{2}}{a_{k-1}^{2}}\left(  x_{2k-3}^{2}+x_{2k-2}^{2}\right)
=(k-1)\frac{b_{k}^{2}}{a_{k}^{2}}\left(  x_{2k-1}^{2}+x_{2k}^{2}\right)  .
\]

\end{corollary}

\qquad Now, we examine the following examples of the non-null self-similar
curves in the Minkowski plane and Minkowski space-time.

\textbf{Case }$n=2$ (Lorentzian plane)$:$ Let $\alpha_{2}:I\rightarrow
\mathbb{E}_{1}^{2}$ be a timelike self-similar curve with the constant
invariant $\tilde{\kappa}_{1}.$ Then, we can write $\lambda_{1}^{2}=1$ and
$a_{1}^{2}\lambda_{1}^{2}=1.$ Hence, a parametrization of $\alpha_{2}$ is
\[
\alpha_{2}\left(  \sigma\right)  =\left(  \frac{1}{\sqrt{1-\tilde{\kappa}%
_{1}^{2}}}e^{\tilde{\kappa}_{1}\sigma}\sinh\theta,\frac{1}{\sqrt
{1-\tilde{\kappa}_{1}^{2}}}e^{\tilde{\kappa}_{1}\sigma}\cosh\theta\right)
\]
where $\theta=\sigma+\cosh^{-1}\left(  \frac{1}{\sqrt{1-\tilde{\kappa}_{1}%
^{2}}}\right)  .$

\textbf{Case }$n=4$ (Minkowski space-time)$:$ Let $\alpha_{4}:I\rightarrow
\mathbb{E}_{1}^{4}$ be a timelike self-similar curve with the constant
invariants $\tilde{\kappa}_{1},$ $\tilde{\kappa}_{2}$ and $\tilde{\kappa}%
_{3}.$ Then, the semi-symmetric matrix
\[
M^{2}=%
\begin{bmatrix}
1 & 0 & \tilde{\kappa}_{2} & 0\\
0 & 1-\tilde{\kappa}_{2}^{2} & 0 & \tilde{\kappa}_{2}\tilde{\kappa}_{3}\\
-\tilde{\kappa}_{2} & 0 & -\tilde{\kappa}_{2}^{2}-\tilde{\kappa}_{3}^{2} & 0\\
0 & \tilde{\kappa}_{2}\tilde{\kappa}_{3} & 0 & -\tilde{\kappa}_{3}^{2}%
\end{bmatrix}
\]
has two eigenvalues of multiplicity 2
\[
\lambda_{i}^{2}=\frac{1}{2}\left(  1-\tilde{\kappa}_{2}^{2}-\tilde{\kappa}%
_{3}^{2}+\left(  -1\right)  ^{i}\sqrt{\left(  1-\tilde{\kappa}_{2}^{2}%
-\tilde{\kappa}_{3}^{2}\right)  ^{2}+4\tilde{\kappa}_{3}^{2}}\right)
\]
for $i=1,2.$ The solution of system of quadratic equations
\begin{align*}
-a_{1}^{2}+a_{2}^{2}  &  =-1\\
a_{1}^{2}\lambda_{1}^{2}+a_{2}^{2}\lambda_{2}^{2}  &  =1
\end{align*}
is given by
\[
a_{1}=\sqrt{\frac{1+\lambda_{2}^{2}}{\lambda_{1}^{2}+\lambda_{2}^{2}}},\text{
\ \ \ \ \ }a_{2}=\sqrt{\frac{1-\lambda_{1}^{2}}{\lambda_{1}^{2}+\lambda
_{2}^{2}}}%
\]
since we have $1-\lambda_{1}^{2}>0.$ Consequently, the spherical arc-length
parametrization of the timelike self-similar curve is $\alpha_{4}$\ given by
\[
\alpha_{4}\left(  \sigma\right)  =\left(  \frac{a_{1}}{b_{1}}e^{\tilde{\kappa
}_{1}\sigma}\sinh\theta,\frac{a_{1}}{b_{1}}e^{\tilde{\kappa}_{1}\sigma}%
\cosh\theta,\frac{a_{2}}{b_{2}}e^{\tilde{\kappa}_{1}\sigma}\sin\theta
,-\frac{a_{2}}{b_{2}}e^{\tilde{\kappa}_{1}\sigma}\cos\theta\right)
\]
where $b_{1}=\sqrt{\lambda_{1}^{2}-\tilde{\kappa}_{1}^{2}},$ $b_{2}%
=\sqrt{\lambda_{2}^{2}+\tilde{\kappa}_{1}^{2}}$ and $\theta_{1}=\lambda
_{1}\sigma-\cosh^{-1}\left(  \frac{\lambda_{1}}{\sqrt{\lambda_{1}^{2}%
-\tilde{\kappa}_{1}^{2}}}\right)  ,$ $\theta_{2}=\lambda_{2}\sigma+\cos
^{-1}\left(  \frac{\lambda_{2}}{\sqrt{\lambda_{2}^{2}+\tilde{\kappa}_{1}^{2}}%
}\right)  .$

\subsection{Non-null Self-similar Curves in Odd-dimensional Min\-kowski Space}

\qquad It can be given non-null self-similar curves in $\mathbb{E}_{1}^{2k+1}$
with the following theorem.

\begin{theorem}
\label{tek self}Let $\alpha:I\rightarrow\mathbb{E}_{1}^{2k+1}$ be a non-null
self-similar curve with constant p-shape curvatures $\tilde{\kappa}_{1}\neq0,$
$\tilde{\kappa}_{2}\neq0,\cdots,\tilde{\kappa}_{2k}\neq0.$ Suppose that
$\lambda_{1}^{2},$ $\lambda_{2}^{2},\cdots,\lambda_{k}^{2}$ are all different
eigenvalues of the symmetric matrix $M^{2}.$ So,

$\mathbf{i})$ If the unit vector $\mathbf{e}_{1}$ is the timelike vector, then
a parametric statement of the timelike self-similar curve $\alpha$ according
to the arc-length parameter $\sigma$ can be written in the form
\begin{align}
\alpha\left(  \sigma\right)   &  =(\frac{a_{1}}{b_{1}}e^{\tilde{\kappa}%
_{1}\sigma}\sinh\theta_{1},\frac{a_{1}}{b_{1}}e^{\tilde{\kappa}_{1}\sigma
}\cosh\theta_{1},\nonumber\\
&  \frac{a_{2}}{b_{2}}e^{\tilde{\kappa}_{1}\sigma}\sin\theta_{2},-\frac{a_{2}%
}{b_{2}}e^{\tilde{\kappa}_{1}\sigma}\cos\theta_{2},\cdots,\frac{a_{k}}{b_{k}%
}e^{\tilde{\kappa}_{1}\sigma}\sin\theta_{k},-\frac{a_{k}}{b_{k}}%
e^{\tilde{\kappa}_{1}\sigma}\cos\theta_{k},a_{k+1}e^{\tilde{\kappa}_{1}\sigma
})
\end{align}
where
\[
b_{1}=\sqrt{\lambda_{1}^{2}-\tilde{\kappa}_{1}^{2}},\text{ \ \ \ \ \ \ }%
\theta_{1}=\lambda_{1}\sigma-\cosh^{-1}\left(  \frac{\lambda_{1}}%
{\sqrt{\lambda_{1}^{2}-\tilde{\kappa}_{1}^{2}}}\right)  ,
\]
and for $i=2,\cdots,k$
\[
b_{i}=\sqrt{\lambda_{i}^{2}+\tilde{\kappa}_{1}^{2}},\text{ \ \ \ \ \ \ }%
\theta_{i}=\lambda_{i}\sigma+\cos^{-1}\left(  \frac{\lambda_{i}}{\sqrt
{\lambda_{i}^{2}+\tilde{\kappa}_{1}^{2}}}\right)  .
\]
The real different non-zero numbers $a_{1},\cdots,$ $a_{k},$ $a_{k+1}$ are a
solution of the system of $k+1$ algebraic quadratic equations
\begin{align*}
\mathbf{e}_{1}\cdot\mathbf{e}_{1}  &  =-1\text{ \ \ \ \ or \ \ \ \ }-a_{1}%
^{2}+\sum\limits_{i=2}^{k}a_{i}^{2}+\tilde{\kappa}_{1}^{2}a_{k+1}^{2}=-1\\
\mathbf{e}_{2}\cdot\mathbf{e}_{2}  &  =1\text{ \ \ \ \ \ \ or \ \ \ \ }%
\sum\limits_{i=1}^{k}a_{i}^{2}\lambda_{i}^{2}=1\\
\mathbf{e}_{3}\cdot\mathbf{e}_{3}  &  =1\text{ \ \ \ \ \ \ or \ \ \ \ }%
-a_{1}^{2}\left(  1-\lambda_{1}^{2}\right)  ^{2}+\sum\limits_{i=2}^{k}\left(
1+\lambda_{i}^{2}\right)  ^{2}a_{i}^{2}+\tilde{\kappa}_{1}^{2}a_{k+1}%
^{2}=\tilde{\kappa}_{2}^{2}\\
\mathbf{e}_{i}\cdot\mathbf{e}_{i}  &  =1,\text{ \ \ \ \ \ }i=4,\cdots,k+1
\end{align*}
determined by the vectors
\begin{align*}
\mathbf{e}_{1}\left(  \sigma\right)   &  =e^{-\tilde{\kappa}_{1}\sigma}%
\frac{d}{d\sigma}\alpha\left(  \sigma\right) \\
\mathbf{e}_{2}\left(  \sigma\right)   &  =\frac{d}{d\sigma}\mathbf{e}%
_{1}\left(  \sigma\right) \\
\mathbf{e}_{3}\left(  \sigma\right)   &  =\frac{1}{\tilde{\kappa}_{2}}\left(
-\mathbf{e}_{1}\left(  \sigma\right)  +\frac{d}{d\sigma}\mathbf{e}_{2}\left(
\sigma\right)  \right) \\
\mathbf{e}_{4}\left(  \sigma\right)   &  =\frac{1}{\tilde{\kappa}_{3}}\left(
\tilde{\kappa}_{2}\mathbf{e}_{2}\left(  \sigma\right)  +\frac{d}{d\sigma
}\mathbf{e}_{3}\left(  \sigma\right)  \right) \\
&  \vdots\text{ \ \ \ \ \ \ \ \ \ \ \ \ \ \ \ }\vdots\\
\mathbf{e}_{k+1}\left(  \sigma\right)   &  =\frac{1}{\tilde{\kappa}_{k}%
}\left(  \tilde{\kappa}_{k-1}\mathbf{e}_{k-1}\left(  \sigma\right)  +\frac
{d}{d\sigma}\mathbf{e}_{k}\left(  \sigma\right)  \right)  .
\end{align*}

$\mathbf{ii})$ If the unit vector $\mathbf{e}_{2}$ is timelike vector, then a
parametric representation of the spacelike self-similar curve $\alpha$
according to arc-length parameter $\sigma$ can be written in the form
\begin{align}
\alpha\left(  \sigma\right)   &  =(-\frac{a_{1}}{b_{1}}e^{\tilde{\kappa}%
_{1}\sigma}\cosh\theta_{1},-\frac{a_{1}}{b_{1}}e^{\tilde{\kappa}_{1}\sigma
}\sinh\theta_{1},\nonumber\\
&  \frac{a_{2}}{b_{2}}e^{\tilde{\kappa}_{1}\sigma}\sin\theta_{2},-\frac{a_{2}%
}{b_{2}}e^{\tilde{\kappa}_{1}\sigma}\cos\theta_{2},\cdots,\frac{a_{k}}{b_{k}%
}e^{\tilde{\kappa}_{1}\sigma}\sin\theta_{k},-\frac{a_{k}}{b_{k}}%
e^{\tilde{\kappa}_{1}\sigma}\cos\theta_{k},a_{k+1}e^{\tilde{\kappa}_{1}\sigma
})
\end{align}
The real different non-zero numbers $a_{1},\cdots,$ $a_{k},$ $a_{k+1}$ are a
solution of the system of $k+1$ algebraic quadratic equations
\begin{align*}
\mathbf{e}_{1}\cdot\mathbf{e}_{1}  &  =1\text{ \ \ \ \ or \ \ \ \ }%
\sum\limits_{i=1}^{k}a_{i}^{2}+\tilde{\kappa}_{1}^{2}a_{k+1}^{2}=1\\
\mathbf{e}_{2}\cdot\mathbf{e}_{2}  &  =-1\text{ \ \ or \ \ \ \ }-a_{1}%
^{2}+\sum\limits_{i=2}^{k}a_{i}^{2}\lambda_{i}^{2}=-1\\
\mathbf{e}_{3}\cdot\mathbf{e}_{3}  &  =1\text{ \ \ \ \ or \ \ \ \ }a_{1}%
^{2}\left(  1-\lambda_{1}^{2}\right)  ^{2}+\sum\limits_{i=2}^{k}\left(
1+\lambda_{i}^{2}\right)  a_{i}^{2}+\tilde{\kappa}_{1}^{2}a_{k+1}^{2}%
=\tilde{\kappa}_{2}^{2}\\
\mathbf{e}_{i}\cdot\mathbf{e}_{i}  &  =1,\text{ \ \ \ \ \ \ \ \ \ \ \ }%
i=4,\cdots,k+1,
\end{align*}
determined by the vectors
\begin{align*}
\mathbf{e}_{1}\left(  \sigma\right)   &  =e^{-\tilde{\kappa}_{1}\sigma}%
\frac{d}{d\sigma}\alpha\left(  \sigma\right) \\
\mathbf{e}_{2}\left(  \sigma\right)   &  =-\frac{d}{d\sigma}\mathbf{e}%
_{1}\left(  \sigma\right) \\
\mathbf{e}_{3}\left(  \sigma\right)   &  =\frac{1}{\tilde{\kappa}_{2}}\left(
\mathbf{e}_{1}\left(  \sigma\right)  +\frac{d}{d\sigma}\mathbf{e}_{2}\left(
\sigma\right)  \right) \\
\mathbf{e}_{4}\left(  \sigma\right)   &  =\frac{1}{\tilde{\kappa}_{3}}\left(
-\tilde{\kappa}_{2}\mathbf{e}_{2}\left(  \sigma\right)  +\frac{d}{d\sigma
}\mathbf{e}_{3}\left(  \sigma\right)  \right) \\
&  \vdots\text{ \ \ \ \ \ \ \ \ \ \ \ \ \ \ \ }\vdots\\
\mathbf{e}_{k}\left(  \sigma\right)   &  =\frac{1}{\tilde{\kappa}_{k-1}%
}\left(  \tilde{\kappa}_{k-2}\mathbf{e}_{k-2}\left(  \sigma\right)  +\frac
{d}{d\sigma}\mathbf{e}_{k-1}\left(  \sigma\right)  \right)  .
\end{align*}

\end{theorem}

\begin{proof}
The proof is the same as the proof of Theorem $\ref{cift self}.$
\end{proof}

\begin{corollary}
$\mathbf{i})$ Let $\alpha:I\rightarrow\mathbb{E}_{1}^{2k+1}$ $(k>1)$ be a
timelike self-similar curve with a parametric representation $\left(
\ref{eq}\right)  .$ Then, this curve lies on the quadratic timelike
hypersurface with an equation%
\[
\frac{b_{1}^{2}}{a_{1}^{2}}\left(  -x_{1}^{2}+x_{2}^{2}\right)  +\frac
{b_{2}^{2}}{a_{2}^{2}}\left(  x_{1}^{2}+x_{2}^{2}\right)  +\cdots
+\frac{b_{k-1}^{2}}{a_{k-1}^{2}}\left(  x_{2k-3}^{2}+x_{2k-2}^{2}\right)
+\frac{b_{k}^{2}}{a_{k}^{2}}\left(  x_{2k-1}^{2}+x_{2k}^{2}\right)  =\frac
{k}{a_{k+1}^{2}}x_{2k+1}^{2}.
\]

$\mathbf{ii})$ Let $\alpha:I\rightarrow\mathbb{E}_{1}^{2k}$ $(k>1)$ be a
spacelike self-similar curve with a parametric representation $\left(
\ref{spa}\right)  .$ Then, this curve lies on the quadratic Lorentzian
hypersurface with an equation%
\[
\frac{b_{1}^{2}}{a_{1}^{2}}\left(  x_{1}^{2}-x_{2}^{2}\right)  +\frac
{b_{2}^{2}}{a_{2}^{2}}\left(  x_{1}^{2}+x_{2}^{2}\right)  +\cdots
+\frac{b_{k-1}^{2}}{a_{k-1}^{2}}\left(  x_{2k-3}^{2}+x_{2k-2}^{2}\right)
+\frac{b_{k}^{2}}{a_{k}^{2}}\left(  x_{2k-1}^{2}+x_{2k}^{2}\right)  =\frac
{k}{a_{k+1}^{2}}x_{2k+1}^{2}.
\]

\end{corollary}

\qquad Now, we investigate the timelike self-similar curves in the Minkowski
3-space $\mathbb{E}_{1}^{3}$.

\textbf{Case }$n=3$ (Minkowski 3-space)$:$ Let $\alpha_{3}:I\rightarrow
\mathbb{E}_{1}^{3}$ be a timelike self-similar curve with constant invariant
$\tilde{\kappa}_{1}\neq0$ and $\tilde{\kappa}_{2}\neq0.$ Then, the
semi-symmetric matrix
\[
M^{2}=%
\begin{bmatrix}
1 & 0 & \tilde{\kappa}_{2}\\
0 & 1-\tilde{\kappa}_{2}^{2} & 0\\
-\tilde{\kappa}_{2} & 0 & -\tilde{\kappa}_{2}^{2}%
\end{bmatrix}
\]
has a unique non-zero eigenvalue of multiplicity 2
\[
\lambda_{1}^{2}=1-\tilde{\kappa}_{2}^{2}%
\]
and therefore $b_{1}=\sqrt{1-\tilde{\kappa}_{1}^{2}-\tilde{\kappa}_{2}^{2}}$.
By the Theorem $\ref{tek self},$ we can compute
\[
a_{1}=\sqrt{\frac{1}{1-\tilde{\kappa}_{2}^{2}}},\text{ \ \ \ \ \ \ }%
a_{2}=\sqrt{\frac{\tilde{\kappa}_{2}^{2}}{\tilde{\kappa}_{1}^{2}\left(
1-\tilde{\kappa}_{2}^{2}\right)  }}.
\]
Hence, a parametrization of $\alpha_{3}$ with respect to spherical arc-lentgh
is given by%
\begin{equation}
\alpha_{3}\left(  \sigma\right)  =\left(  \frac{e^{\tilde{\kappa}_{1}\sigma
}\sinh\theta}{\sqrt{\left(  1-\tilde{\kappa}_{2}^{2}\right)  \left(
1-\tilde{\kappa}_{1}^{2}-\tilde{\kappa}_{2}^{2}\right)  }},\frac
{e^{\tilde{\kappa}_{1}\sigma}\cosh\theta}{\sqrt{\left(  1-\tilde{\kappa}%
_{2}^{2}\right)  \left(  1-\tilde{\kappa}_{1}^{2}-\tilde{\kappa}_{2}%
^{2}\right)  }},\sqrt{\frac{\tilde{\kappa}_{2}^{2}}{\tilde{\kappa}_{1}%
^{2}\left(  1-\tilde{\kappa}_{2}^{2}\right)  }}e^{\tilde{\kappa}_{1}\sigma
}\right)  \label{alfa3}%
\end{equation}
where $\theta=\sigma\sqrt{1-\tilde{\kappa}_{2}^{2}}+\cosh^{-1}\sqrt
{\frac{1-\tilde{\kappa}_{1}^{2}}{1-\tilde{\kappa}_{1}^{2}-\tilde{\kappa}%
_{2}^{2}}}.$ It is clear that the timelike self-similar curve $\alpha_{3}$ is
a curve on the surface with an implicit equation $\frac{\tilde{\kappa}_{2}%
^{2}}{\tilde{\kappa}_{1}^{2}}\left(  -x_{1}^{2}+x_{2}^{2}\right)  =\frac
{1}{1-\tilde{\kappa}_{1}^{2}-\tilde{\kappa}_{2}^{2}}x_{3}^{2}$ (see figure 2).%

\begin{figure}
[ptb]
\begin{center}
\includegraphics[
height=3.1631in,
width=4.2in
]%
{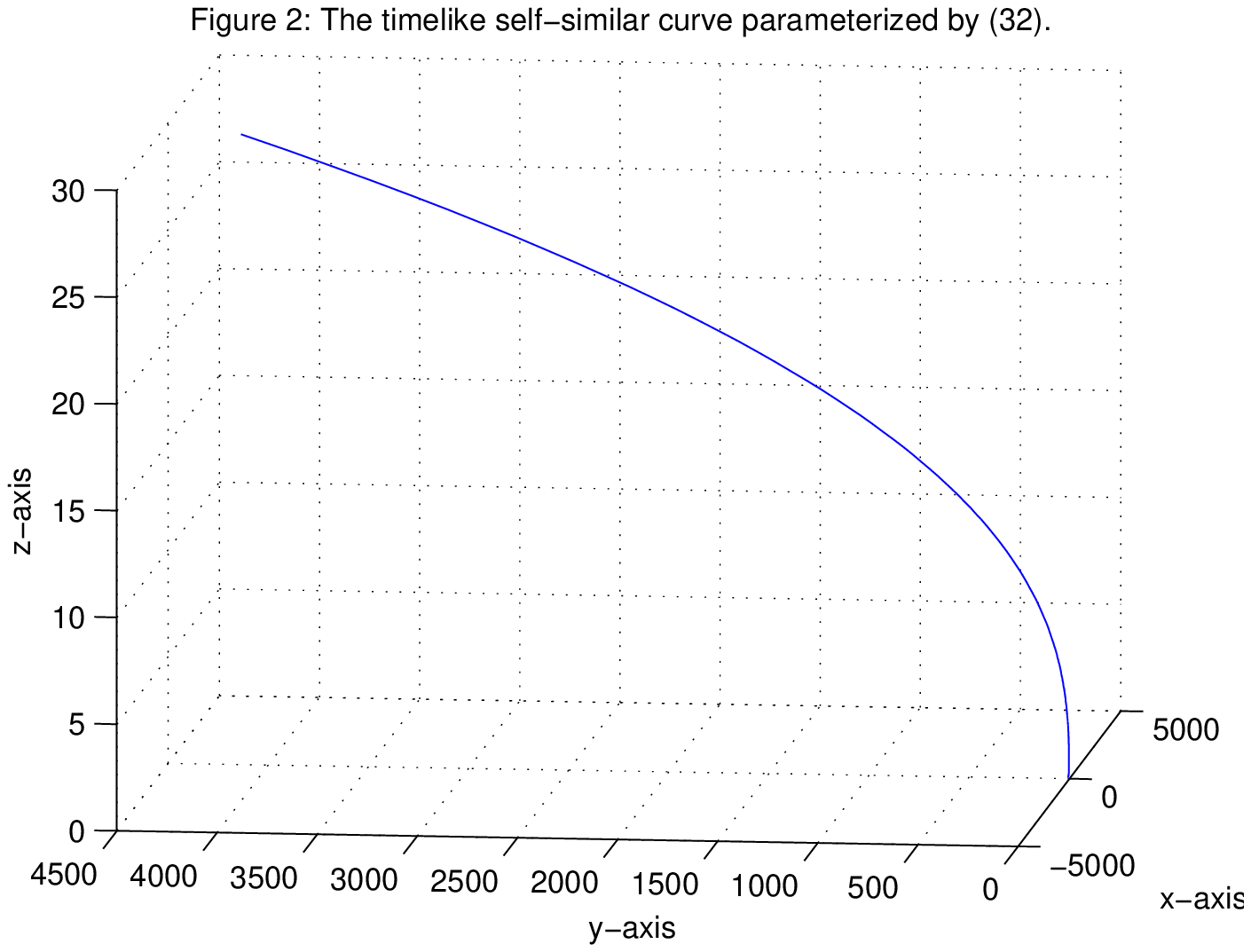}%
\end{center}
\end{figure}

\section{Concluding Remarks}

\qquad In this paper, we gave the p-similarity invariants of a non-null curves
in the Lorentzian n-space. We also proved the fundamental existence and
uniqueness theorems for non-null curves under p-similarity transformation. We
studied self-similar non-null curves in Lorentzian n-space. We think that the
notion of similarity and self-similarity in the Lorentzian-Minkowski space may
form many new ideas and concepts to the pure and applied mathematics.

\qquad Bejancu \cite{null0} represented a method for the general study of the
geometry of null curves in Lorentz manifolds and, more generally, in
semi-Riemannian manifolds (see also \cite{duggal}). A. Ferrandez, A. Gimenez,
and P. Lucas \cite{null1} generalized the Cartan frame to Lorentzian space
forms. They showed the fundamental existence and uniqueness theorems and they
obtained values of the Cartan curvatures in higher dimensions. Therefore, it
is of interest\ to investigate the similarity invariants of null curves in
Lorentzian n-space under these considerations.

\qquad The motions of curves in $\mathbb{E}^{2},$ $\mathbb{E}^{3}$ and
$\mathbb{E}^{n}$ $(n>3)$ yield the mKdV hierarchy, Schr\"{o}dinger hierarchy
and a multi- component generalization of mKdV-Schr\"{o}dinger hierarchies,
respectively. KS. Chou and C. Qu \cite{chaos} showed that the motions of
curves in two-, three- and n-dimensional $(n>3)$ similarity geometries
correspond to the Burgers hierarchy, Burgers-mKdV hierarchy and a
multi-component generalization of these hierarchies by using the similarity
invariants of curves in comparison with its invariants under the Euclidean
motion. Also, they \cite{chaos0} found that many 1+1-dimensional integrable
equations like KdV, Burgers, Sawada-Kotera, Harry-Dym hierarchies and
Camassa-Holm equations arise from motions of plane curves in centro-affine,
similarity, affine and fully affine geometries. The motion of curves on
two-dimensional surfaces in $\mathbb{E}_{1}^{3}$ was considered by G\"{u}rses
\cite{gurses}. Q. Ding and J. Inoguchi \cite{inoguchi} showed binormal motions
of curves in Minkowski 3-space are closely related to Schr\"{o}dinger flows
into the Lorentzian symmetric space and Riemannian symmetric space. Therefore,
with the aid of the current paper, it will be studied the motion of Lorentzian
similar curves with p-similarity invariants under the consideration of the
paper \cite{chaos}.

\bigskip

Hakan Sim\c{s}ek and Mustafa \"{O}zdemir

Department of Mathematics

Akdeniz University

Antalya, TURKEY;

e-mail: hakansimsek@akdeniz.edu.tr.

\ \ \ \ \ \ \ \ \ \ \ mozdemir@akdeniz.edu.tr
\end{document}